\documentclass[11pt]{article}
\usepackage{amsmath,amsfonts,amssymb}
\usepackage{latexsym}
\usepackage{epsfig,overpic}
\usepackage{subfigure}
\usepackage{color}
\usepackage{amsthm}

\newcommand{\E}{\mathcal E}

\newcommand{\T}{\mathcal T}

\DeclareGraphicsExtensions{.pdf,.eps,.ps,.eps.gz,.ps.gz,.eps.Y}

\def\dx{\mathrm{d}x}

\newcommand{\eps}{\varepsilon}

\newtheorem{remark}{Remark}[section]
\newtheorem{theorem}{Theorem}[section]
\newtheorem{lemma}{Lemma}[section]

\begin{document}
\title{A variational analysis for the moving finite element method for gradient flows}
\author{ Xianmin Xu\thanks{LSEC, Institute of Computational Mathematics and Scientific/Engineering Computing,
  NCMIS, AMSS, Chinese Academy of Sciences, Beijing 100190, China; School of Mathematical Sciences, University of Chinese Academy of Sciences, Beijing 100049, China (xmxu@lsec.cc.ac.cn).
  }
}
\maketitle
\begin{abstract} 
By using the Onsager principle as an approximation tool, we give a novel  derivation
for the moving finite element method for gradient flow equations. 
We show that the discretized problem has the same energy dissipation structure as the continuous one.
This enables us to do numerical analysis for the stationary solution of a nonlinear reaction diffusion equation
using the  approximation theory of free-knot  piecewise polynomials.
We show that under certain  conditions the solution obtained by the moving finite element method 
converges to a local minimizer of the total energy when time goes to infinity.
The global minimizer, once it is detected by the discrete scheme, approximates 
 the continuous stationary solution in optimal order.
Numerical examples  for a  linear diffusion equation and a nonlinear Allen-Cahn equation
are given to verify the analytical results.
\end{abstract}
%
\section{Introduction}
The moving finite element method(MFEM) was first developed  in \cite{miller1981movingA,miller1981movingB} about forty years ago.
It  is a typical  $r$-type adaptive method\cite{huang1994moving,cao2002moving,li2001moving,tang2005moving,budd2009adaptivity},
where the mesh vertexes are relocated without changing the mesh topology.  
In the MFEM, the mesh relocation
is done by solving a dynamic equation for the vertexes
 coupled with the original partial differential equations. 
No interpolation is needed in the method since the mesh is continuous with time.
The MFEM has arisen considerable interest and has been further developed in several directions(c.f \cite{Baines,dupont1982mesh,carlson1998designA,wathen1985structure,baines2005moving,bank2019diagonally} among many others). 

However, like  all other $r$-adaptive methods, the theoretical analysis for the MFEM is far from being complete. 
The first error analysis for  MFEM was done by Dupont \cite{dupont1982mesh}, where he proved the optimal convergence
of the method
for a linear convection diffusion equation when the solution is smooth. This is not enough since we are more interested in
 non-smooth solutions for adaptive methods. 
Later on, Jimack proved the locally optimal approximation for the stationary solution of a linear parabolic equation\cite{jimack1992steady,jimack1996best,jimack1996optimal}. This is remarkable since 
the solution is allowed to have lower regularity  in Jimack's proof. 
In this study, we aim to do analysis for a nonlinear gradient flow system by using the Onsager variational principle as an approximation tool. 

The Onsager variational principle is a fundamental principle for irreversible thermodynamic processes
in statistical physics\cite{Onsager1931,Onsager1931a,DoiSoft}. It has be 
used to model many dissipative physical systems\cite{Doi2011,DoiSoft},
 such as  the Stokes equation in hydrodynamics, the Ericksen-Leslie equation in liquid
crystal, and the GNBC boundary condition for moving contact lines\cite{QianTiezheng2006}, etc.
Recently, the Onsager principle has been used as an approximation tool for many problems in two-phase flows
and in material science\cite{Doi2015,XuXianmin2016,DiYana2016,ManXingkun2017,ZhouJiajia2018,Jiang19b}.
In particular, it has  been used to derive an efficient numerical method for wetting dynamics\cite{LuXu2020}.

In this work, we first give a new derivation of the MFEM  for a gradient flow system by using the Onsager principle as an approximation tool.
 The key idea is to approximate the system in a nonlinear approximation space of free-knot piecewise polynomials\cite{devore1998nonlinear}.
Both the mesh vertexes and the nodal values of the finite element function  are regarded as  unknowns.
We  derive a system of ordinary differential equations(ODEs) for them. The ODE system  coincides
with the discrete equation of the MFEM, which has been derived in a totally different way in \cite{miller1981movingA}.
Here we do not need to compute the multiply of a Dirac measure and a discontinuous function,
so that the ``mollification" technique or any other formally calculation is not needed. 
Furthermore, our derivation shows that the discretized problem has the same energy dissipation structure of
the continuous system. This makes us to prove the energy decay property of the discrete problem easily.

 Based on the variational formula, we do  error analyse for
the MFEM for a stationary solution of the gradient flow system.
The analysis can be regarded as a generalization of the results in \cite{jimack1996best} to nonlinear equations.
We show that the MFEM gives locally best approximations to the energy. 
When a global minimizer is detected, an optimal error estimate is proved using the nonlinear
approximation theory. 
Numerical examples show that the optimal convergence  
can be obtained for a linear diffusion equation and for  stationary solutions of a nonlinear Allan-Cahn equation. 
In this paper, we mainly consider the one dimensional problem. All the results can be generalized 
to higher dimensional cases directly.


The structure of the paper is as follows. In section 2, we introduce the Onsager variational principle and show
that it can be used to derive the partial differential equation model for a gradient flow system.
In Section 3, we derive the MFEM by using the Onsager principle as an approximation tool. 
In Section 4, we do error analysis for the stationary solution of a nonlinear reaction diffusion equation.
Some numerical examples are illustrated to verify the analytical results in the last section.


\section{The Onsager variational principle for a gradient flow system}
\subsection{The Onsager principle}
Suppose a physical system is described by a time dependent function $u$.
For simplicity, we denote by $\dot{u}=\frac{\partial u}{\partial t}$ the time derivative of $u$. 
For a dissipated system, the evolution of $u$ can dissipate energy.
The dissipation function is defined as half of the total energy dissipated with respect to the flux $\dot{u}$(c.f. \cite{DoiSoft}).
For a simple gradient flow system, we assume the dissipation function is 
\begin{equation}\label{e:diss}
\Phi(\dot{u})=\frac{\xi}{2}\|\dot{u}\|^2,
\end{equation}
where $\xi$ is a positive friction coefficient and $\|\cdot\|$ is a $L^2$ norm. 
Suppose that the free energy of the system is given by a functional $\E(u)$.
For given $\dot u$, the rate of change of the total energy is calculated by
\begin{equation}\label{e:engRate}
\dot{\E}(u;\dot{u})=\langle \frac{\delta \E(u)}{\delta u},\dot{u}\rangle.
\end{equation}
Then  a Rayleighian (functional)  is defined as 
\begin{equation}\label{e:Reyl}
\mathcal{R}(u;\dot{u})=\Phi(\dot{u})+\dot{\E}(u;\dot{u}).
\end{equation}

With these definitions, the Onsager  principle can be stated as follows(\cite{DoiSoft}). For 
any given $u$ at the present time, the time derivative $\dot{u}$ is obtained by minimizing the Rayleighian
among its all possible choices. In other words, the  evolution equation of $u$ is determined by 
minimizing the Rayleighian $\mathcal{R}(u,\dot{u})$ with respect to $\dot{u}$, i.e.
\begin{equation}\label{e:Onsager}
\min_{\dot{u}\in V} \mathcal{R}(u,\dot{u}).
\end{equation}
Here $V$ is the admissible  space of $\dot{u}$.
Since the Rayleighian is a quadratic form with respect to $\dot{u}$,  the problem~\eqref{e:Onsager} is equivalent to  its 
Euler-Lagrange equation
\begin{equation}\label{e:gf1}
\xi (\dot{u},\psi) =- \langle\frac{\delta \E(u)}{\delta u},{\psi}\rangle,\qquad\qquad \forall \psi\in V,
\end{equation}
or in a simple form,
\begin{equation}\label{e:forcebalance}
\xi\dot{u}=-\frac{\delta \E(u)}{\delta u}.
\end{equation}
 In physics, the equation indicates a balance between the friction force $-\xi\dot{u}$ and
 the general driven force $\frac{\delta \E(u)}{\delta u}$.
The equation \eqref{e:forcebalance}  can be simply rewritten as a gradient flow equation,
\begin{equation}\label{e:gf}
\frac{\partial u }{\partial t}=-\xi^{-1} \frac{\delta \mathcal \E(u)}{\delta u}.
\end{equation}
%

It is easy to see that the solution $u$ of Equation \eqref{e:gf1}(or \eqref{e:gf}) satisfies  the following energy decay property
\begin{equation}\label{e:engdecay}
\frac{d \E(u)}{d t}\leq 0.
\end{equation}
Actually, by setting $\psi=\frac{\partial u}{\partial t}$ in \eqref{e:gf1}, we have
\begin{align}\label{e:engdecay1}
\frac{d \E}{d t} =\langle \frac{\delta \E(u)}{\delta u},\frac{\partial u }{\partial t} \rangle=-\xi \|\frac{\partial u }{\partial t}\|_{V}^2=-2\Phi(\dot{u})
\leq 0.
\end{align}
The rate of decreasing of the total energy is equal to twice of the dissipation function.

\subsection{The model problem}
For simplicity in presentations,  we mainly consider a specific gradient flow system in this paper,
which corresponds to a (nonlinear) reaction diffusion equation. 

 Denote by  $\Omega$ a domain in $R^n$.
Suppose the energy functional is given by 
\begin{equation}\label{e:Eng}
\E(u,\nabla u)=\int_{\Omega} \frac{\alpha}{2}(\nabla u)^2+F(x,u) dx,
\end{equation}
where $\alpha>0$ is the diffusion coefficient and $F(x,u)$ is a function with respect to $x$ and $u$. 
We assume that $u\in H^1_0(\Omega)$ and $\dot u\in L^2(\Omega)$. 
The dissipation function is given by
\begin{equation}
\Phi(\dot{u})=\frac{\xi}{2}\int_{\Omega} \dot u^2 dx,
\end{equation}
with a positive friction coefficient $\xi$.  Then the Rayleighian is calculated as 
\begin{align}
\mathcal{R}(u,\dot u)&=\Phi(\dot u)+\dot{\E}=\frac{\xi}{2}\int_{\Omega} \dot u^2 dx+ \int_{\Omega} \alpha \nabla u\cdot \nabla \dot u+ \partial_u F(x,u) \dot udx.
\label{e:Rayleig1}
\end{align}
We minimize $\mathcal{R}$ with respect to $\dot u$.

By integral by part, $\mathcal{R}$ is rewritten as 
\begin{align}\nonumber
\mathcal{R}(u,\dot u)&=\frac{\xi}{2}\int_{\Omega} \dot u^2 dx+ \int_{\Omega} [-\alpha\Delta u+\partial_u F(x,u) ]\dot udx.
\end{align}
The corresponding Euler-Langrange equation of \eqref{e:Rayleig1}  is 
\begin{equation}\label{e:modelpb}
\partial_t u -\frac{\alpha}{\xi}\Delta u+\frac{1}{\xi}  f(x,u) =0,
\end{equation}
where $f(x,u)=\partial_u F(x,u)$.
It is a quasi-linear reaction diffusion equation.
Some typical examples include the linear diffusion equation(when $F=0$) and the Allen-Cahn equation(when $F$ is 
a double-well function), etc.

In general,  we assume that $f(x,u)$ is Liptchitz continuous with respect to $u$, i.e.
\begin{equation}\label{e:Lip}
| f(x,v)- f(x,w)|\leq L_0 |v-w|, \qquad \forall v, w,
\end{equation}
for some constant $L_0>0$.  Under this condition and some assumption on the domain $\Omega$,
the equation \eqref{e:modelpb} has a unique  solution (c.f. Theorem 5.1 in \cite{pao2012nonlinear}).

\section{Derivation of the moving finite element method by the Onsager Principle}

The Onsager principle can be used to derive numerical methods for the model problem~\eqref{e:modelpb}.
The  idea is as follows. We  choose a finite dimensional subspace $V_h$ of $H^1_0(\Omega)$.
For a time dependent function $u_h(t,x)\in V_h$, 
we can compute  the energy functional $\E(u_h)$  and the dissipation function 
$\Phi(\dot{u}_h)$. By minimizing the discrete Rayleighian
$
\mathcal{R}(u_h;\dot{u}_h):=\dot{E}(u_h;\dot{u}_h)+\Phi(\dot{u}_h)
$
with respect to $\dot{u}_h$, we  obtain a dynamic equation for $u_h(t,x)$,
which is a (semi-)discrete problem for \eqref{e:modelpb}. 
For example, if we choose $V_h$ as a standard Lagrangian finite element space,  it is easy to verify that we 
can obtain a finite element problem same as that derived by the standard Galerkin approach.
In the following, we  choose $V_h$ as a nonlinear approximation space composed of
free-knot piecewise polynomials. Then we derive the MFEM proposed in \cite{miller1981movingA}.

For simplicity in notations, we consider only the one-dimensional problem hereinafter.
The domain $\Omega$ is simply  a bounded interval $I=(a,b)$.  The model problem~\eqref{e:modelpb} is
reduced to
\begin{equation}\label{e:modelpb1}
\partial_t u -\frac{\alpha}{\xi}\partial_{xx} u+\frac{1}{\xi}  f(x,u) =0.
\end{equation}

\subsection{Nonlinear  approximation space}
We  recall some known results on the nonlinear approximation space of the piecewise linear functions with free-knots \cite{devore1998nonlinear}.
Let $N$ be a positive integer and let 
$${X}:=\{a=:x_0< x_1<\cdots<x_N:=b\}$$
be a set   of ordered points in the interval $\bar{I}=[a,b]$. This generates a partition of $I$, 
which we denote as
$\mathcal{T}({X}):=\{I_k\}_{k=1}^N$, where $I_{k}=(x_{k-1},x_k)$. Let 
\begin{equation}\label{e:finiteElement}
V_h(X):=\{ v_h\in C([a,b]) | v_h \hbox{ is linear in }  I_k, \forall k=1,\cdots, N \},
\end{equation}
be the standard piecewise linear finite element space with respect to the partition $\T(X)$.
Denote by ${V}_h^N$ the space with $N$ intervals as follows
\begin{equation}\label{e:Nspace}
V_h^N:=\cup_{\#(X)=N+1} V_h(X),
\end{equation}
where $\#(X)$ denotes the cardinality of $X$. It is a function space for piecewise linear
functions with free knots $\{x_k\}$.
Notice that $V_h^N$ is not 
a linear  space, since the summation of two functions in $V_h^N$ may 
not belongs to the same space when they correspond to different partitions $X$.

For any function $u_h(x)\in V_h^N$, it can be written as
\begin{equation}\label{e:basis}
u_h(x)=\sum_{k=1}^{N-1} u_k\phi_k(x),
\end{equation}
where $\phi_k(x)$ is the standard nodal basis function with respect to a partition 
$\mathcal{T}(X)$, namely
\begin{equation*}
\phi_k(x)=\frac{x-x_{k-1}}{x_{k}-x_{k-1}}\chi_{I_k}(x)+\frac{x_{k+1}-x}{x_{k+1}-x_{k}} \chi_{I_{k+1}}(x),
\end{equation*}
where $\chi_{I_k}$ is the characteristic function corresponding to $I_k$,
\begin{equation*}
\chi_{I_k}(x)=\left\{
\begin{array}{ll}
1& \hbox{if } x \in I_{k} \\
0& \hbox{otherwise}.
\end{array}
\right.
\end{equation*}
In Equation~\eqref{e:basis}, both $u_k$ and $x_k$ ($k=1,\cdots,N-1$) can change their values.
 Therefore,
 $V_h^N$ is a $2(N-1)$ manifold. 
We remark that $V_h^N$ is not a smooth manifold in general. There exists some function $u_h\in V_h^N$
which corresponds to many different coordinates $\{u_k, x_k\ |\ k=1,\cdots N-1\}$. 
 In this sense, the manifold 
 degenerates for some functions in $V_h^N$. 

The nonlinear approximation theory for free-knot piecewise linear polynomials has been studied extensively (c.f.  \cite{petrushev1988direct,devore1998nonlinear}
and the reference therein). 
For a function $u\in H^1_0$, the best approximation of $u$ in $V_h^N$ in energy norm is defined as
\begin{equation}
\sigma_N(u)=\inf_{v_h\in V_h^N}|u-v_h|_{H^1}.
\end{equation}
Here $|u|_{H^1}:=\Big(\int_{I} (\partial_x u)^2 dx\Big)^{\frac{1}{2}}$ is the standard $H^1$ semi-norm.

To characterize the approximation property of functions in $V_h^N$, it is convenient to use the Besov spaces.
We denote by $B_q^r(L_q(I))$ a standard Besov space. 
We will not give the details of the definitions of the Besov space here, but refer to \cite{triebel1978interpolation,devore1993besov}.
We  only mention that $B_q^r(L_q(I))$  is a  space consisting of functions with smoothness order $r$ measured in $L_q$.
When $q=2$, $B_2^r(L_2(I))$ is identical to the Sobolev space $H^r$.
The following lemma is known from the literature(c.f.  \cite{
 binev2002approximation,binev2004adaptive}). 
\begin{lemma} \label{lem:approx}
If $u\in B^{s+1}_q(L_q(I))$ with $0\leq s\leq 1$ and $1/q<s+1/2$, then we have 
\begin{equation}
\sigma_N(u)\leq C N^{-s} |u|_{ B^{s+1}_q(L_q(I))}.
\end{equation}
\end{lemma}
The lemma is a one-dimensional version of  Theorem 9.1 in \cite{binev2004adaptive}.
The proof of the lemma can be found in \cite{binev2002approximation}.
 By this lemma, the best approximation in energy norm of a (one-dimensional) function  by a free-knot piecewise linear function 
is of order $O(N^{-1})$.
\subsection{The moving finite element method}
Consider a time dependent function $u_h(t,x)=\sum_{k=1}^{N-1}u_k(t)\phi_k(t,x)$ in $V_h^N$,
where 
\begin{equation}
\phi_k(t,x)=\frac{x-x_{k-1}(t)}{x_{k}(t)-x_{k-1}(t)}\chi_{I_k}(x)+\frac{x_{k+1}(t)-x}{x_{k+1}(t)-x_{k}(t)} \chi_{I_{k+1}}(x),
\end{equation}
 is the nodal basis function corresponding to a time dependent partition $\mathcal{T}(X(t))$ and the set $X(t)$ is given by
 $$X(t):=\{a=x_0< x_1(t)<\cdots<x_{N-1}(t)<x_N=b\}.$$ 
In the partition, the interior notes $x_k(t), k=1,\cdots, N-1,$ may change positions when time $t$ evolves.
In the formula of $u_h(t,x)$, there are $2(N-1)$ time dependent parameters 
\begin{equation}
\{u_k(t), x_k(t)\ |\ k=1,\cdots N-1\}. 
\end{equation}
We aim to approximate the solution $u$ of the model problem~\eqref{e:modelpb1} by a
discrete function $u_h(t,x)$. For that purpose, we will derive a dynamic equation for $u_k(t)$ and $x_k(t)$
by using the Onsager principle. The derivation is similar to that for continuous problems in Section~2.

We first compute the discrete energy functional and the discrete dissipation function as 
follows. Notice that the time derivative and space derivative of $u_h(t,x)$ are given by
\begin{align*}
\partial_t u_h&=\sum_{k=1}^{N-1} (\dot{u}_k(t)\phi_k(t,x)+\dot{x}_k(t)\beta_k(t,x)),\\
\partial_x u_h&=\sum_{k=1}^{N-1} {u}_k(t) \partial_x \phi_k(t,x), 
\end{align*}
where 
$$\beta_k(t,x)=\frac{\partial u_h}{\partial x_k}=-D_h u_{k-1}\frac{x-x_{k-1}(t)}{x_{k}(t)-x_{k-1}(t)}\chi_{I_k}(x)
-D_h  u_{k}\frac{x_{k+1}(t)-x}{x_{k+1}(t)-x_{k}(t)}\chi_{I_{k+1}}(x),$$
with $D_h  u_k=\frac{u_{k+1}(t)-u_{k}(t)}{x_{k+1}(t)-x_{k}(t)}$. 
Then the discrete energy functional $\E$ with respect to  $u_h$ is calculated by 
\begin{equation}
\E_h(u_1,\cdots,u_{N-1},x_1,\cdots,x_{N-1}):=\E(u_h)=\sum_{k=1}^{N} \int_{I_k}\frac{\alpha}{2}(\partial_x u_h)^2+F(x,u_h) \dx.
\end{equation}
It is a nonlinear function with respect to  $\{u_k(t), x_k(t)\ |\ k=1,\cdots N-1\}$. The discrete dissipation function is given by
\begin{align}
&\Phi_h(u_1,\cdots,u_{N-1},x_1,\cdots,x_{N-1};
\dot u_1,\cdots,\dot{u}_{N-1},\dot{x}_1,\cdots,\dot{x}_{N-1})\nonumber \\
&:=\Phi(\partial_t u_h(t,x))= \frac{\xi}{2}\int_{I}(\dot{u}_k(t)\phi_k(t,x)+\dot{x}_k(t)\beta_k(t,x))^2 dx.\label{e:dissip_h}
\end{align}
It is a quadratic function with respect to  $\{\dot{u}_k(t), \dot{x}_k(t)\ |\ k=1,\cdots N-1\}$.

We apply the Onsager
principle for $\{\dot{u}_k(t), \dot{x}_k(t)\ |\ k=1,\cdots N-1\}$. We minimize the discrete Rayleighian $\mathcal{R}_{h}$ 
with respect to $\dot{u}_k$ and $\dot{x}_k$:
\begin{equation}
\min_{\dot{u}_k,\dot{x}_k} \mathcal{R}_{h}:=\Phi_h+\dot{\E}_h.
\end{equation}
The problem is equivalent to  its Euler-Lagrange equation 
\begin{equation}\label{e:Onsager1}
\left\{
\begin{array}{ll}
\frac{\partial \Phi_h}{\partial \dot{u}_k}+\frac{\partial \E_h}{\partial u_k}=0, &k=1,\cdots, N-1;\\
\frac{\partial \Phi_h}{\partial \dot{x}_k}+\frac{\partial \E_h}{\partial x_k}=0, &k=1,\cdots, N-1.
\end{array}
\right.
\end{equation}
This gives an  ordinary differential system for $\{u_k(t), x_k(t)\ |\ k=1,\cdots N-1\}$.

We now derive the explicit formula for the system \eqref{e:Onsager1}.
Denote  $\mathbf{u}=(u_1,\cdots, u_{N-1})^T$ and 
$\mathbf{x}=(x_1,\cdots, x_{N-1})^T$. 
Notice that $\Phi_h$ is a quadratic function with respect to $\{\dot{u}_k(t),\dot{x}_k(t)\ |\ k=1,\cdots N-1\}$.
Then the equation~\eqref{e:Onsager1} can be rewritten as 
\begin{equation}\label{e:ODE0}
\left(\begin{array}{cc}
\mathbf{A} &\mathbf{B} \\
\mathbf{B} &\mathbf{C}
\end{array}
\right)
\left(
\begin{array}{l}
\dot{\mathbf{u}}\\
\dot{\mathbf{x}}
\end{array}
\right)=
\left(
\begin{array}{l}
{\mathbf{f}}\\
{\mathbf{g}}
\end{array}
\right).
\end{equation}
The right hand side terms are given by
$$\mathbf{f}=(f_1,f_2,\cdots,f_{N-1})^T, \qquad \hbox{with }$$
$$ f_k=-\frac{1}{\xi}\frac{\partial \E_h}{\partial u_k}=-\frac{1}{\xi}\int_{I}\alpha\partial_x u_h \partial_x \phi_k+f(x,u_h)\phi_k dx.
$$
and 
$$\mathbf{g}=(g_1,g_2,\cdots,g_{N-1})^T, \qquad \hbox{with }$$
$$
g_k=-\frac{1}{\xi}\frac{\partial \E_h}{\partial x_k}=-\frac{1}{\xi}\Big(\int_{I}\alpha\partial_x u_h \partial_x \beta_k+f(x,u_h)\beta_k dx-\frac{\alpha}{2}[(\partial_x u_h)^2]|_{x_{k}}\Big),
$$
where the jump $[(\partial_x u_h)^2]|_{x_{k}}:=(\partial_x u_h)^2|_{I_{k}}-(\partial_x u_h)^2|_{I_{k-1}}$.
The  blocks $\mathbf{A}$, $\mathbf{B}$ and $\mathbf{C}$ in the coefficient matrix are $(N-1)\times(N-1)$ tridiagonal matrices, whose
nonzero elements  are functions of $\{u_k\}$ and $\{x_k\}$. 
Direct computations for the nonzero elements of $\mathbf{A}$ give
\begin{align*}
&a_{k,k}=\int_{I_k\cup I_{k+1}} \phi_k^2 dx,\qquad \qquad\qquad\quad \ \  k=1,\cdots,N-1;\\
&a_{k,k+1}=a_{k+1,k}=\int_{I_{k+1}} \phi_k\phi_{k+1} dx, \qquad k=1,\cdots,N-2.
\end{align*}
This implies  $\mathbf{A}$ is the mass matrix in the standard finite element method on
a given triangulation $\mathcal{T}(X)$ with $X=\{a=x_0<x_1<\cdots<x_N=b\}$.
The nonzero elements of the matrix $\mathbf{B}$ are computed as
\begin{align*}
&b_{k,k}=\int_{I_k\cup I_{k+1}} \phi_k\beta_k dx,\qquad \qquad\qquad\qquad\quad \  k=1,\cdots,N-1;\\
&b_{k,k+1}=b_{k+1,k}=\int_{I_{k+1}}\phi_k\beta_{k+1}dx,\qquad\qquad \ \  k=1,\cdots,N-2.
\end{align*}
We can also compute the nonzero elements of the matrix $\mathbf{C}$,
\begin{align*}
&c_{k,k}=\int_{I_k\cup I_{k+1}}\beta_k^2 dx, \qquad \qquad\qquad\qquad\quad\  k=1,\cdots,N-1;\\
&c_{k,k+1}=c_{k+1,k}=\int_{I_{k+1}}\beta_k\beta_{k+1}dx,\qquad \qquad   \!  k=1,\cdots,N-2.
\end{align*}

The ordinary differential system~\eqref{e:ODE0} is exactly the same as the moving finite element scheme  in~\cite{miller1981movingA}.
The  scheme was originally derived through a $L^2$ projection of a partial differential equation in the tangential space of $u_h$.
There one needs to compute the inner product of a Dirac function and a discontinuous function, which is not well-defined even in
a distribution sense. 
Our derivation is much simpler than that in \cite{miller1981movingA}.  In this sense, the Onsager 
principle gives a natural variational framework for the moving finite element method.

\subsection{The stabilized scheme}
The coefficient matrix in \eqref{e:ODE0} may degenerate for some function $u_h\in V_h^N$. This is because
the dissipation function $\Phi_h$ is a semi-positive definite quadratic form with respect to $\dot{u}_k$ and $\dot{x}_k$.
The semi-positive definiteness is related to the degeneracy of  the manifold $V_h^N$.
For example,  if we set
$u_0(t)=u_1(t)=\cdots=u_{N}(t)=0,$
then $u_h(x,t)\equiv 0$ for all possible choice of $X$.
That implies $\partial_t u_h\equiv 0$ and  $\Phi_h\equiv 0$  for some nonzero $\dot{x}_{k}$.

 To overcome the degeneracy of the system \eqref{e:ODE0}, we 
add a stabilized term as follows,
\begin{equation}
\Phi_h^{\delta}=\Phi_h+ \frac{\delta\xi}{2} \sum_{k=1}^{N-1}\dot{x}_k^2,
\end{equation}
where $\delta>0$ is a small stabilization parameter. 
Other stabilizations can also apply, c.f. \cite{miller1981movingB}.
With the modified dissipation function $\Phi^{\delta}_h$,  by the Onsager principle,  
$\dot{u}_k$ and $\dot{x}_k$ are obtained by
\begin{equation}
\min_{\dot{u}_k,\dot{x}_k} \Phi_h^{\delta} +\dot{\E}_h.
\end{equation}
This leads to a modified  system
\begin{equation}\label{e:Onsager2}
\left\{
\begin{array}{ll}
\frac{\partial \Phi_h^\delta}{\partial \dot{u}_k}+\frac{\partial \E_h}{\partial u_k}=0,&k=1,\cdots, N-1;\\
\frac{\partial \Phi_h^\delta}{\partial \dot{x}_k}+\frac{\partial \E_h}{\partial x_k}=0,&k=1,\cdots, N-1.
\end{array}
\right.
\end{equation}
Once again it is an ordinary differential system for $\{u_k(t), x_k(t)\ |\ k=1,\cdots N-1\}$.
The explicit form of the system \eqref{e:Onsager2} is 
\begin{equation}\label{e:ODE}
\left(\begin{array}{cc}
\mathbf{A} &\mathbf{B} \\
\mathbf{B} &\mathbf{C}+\delta \mathbf{I}
\end{array}
\right)
\left(
\begin{array}{l}
\dot{\mathbf{u}}\\
\dot{\mathbf{x}}
\end{array}
\right)=
\left(
\begin{array}{l}
{\mathbf{f}}\\
{\mathbf{g}}
\end{array}
\right).
\end{equation}
The positive definiteness of the coefficient matrix 
$\mathbf{M}_{\delta}:=\left(\begin{array}{cc}
\mathbf{A} &\mathbf{B} \\
\mathbf{B} &\mathbf{C}+\delta \mathbf{I}
\end{array}\right)$
is given by the following lemma.
\begin{lemma}\label{lem:posdef}
For any given $\delta>0$, $X(t)=\{a=x_0<x_1(t)<\cdots<x_{N-1}(t)<x_N=b\}$, and $\{u_k(t) | k=1,\cdots, N-1\}$,
the coefficient matrix $\mathbf{M}_{\delta}$ is positive definite.
\end{lemma}
\begin{proof}
For any $\mathbf{y}\in R^{2(N-1)}$, we can denote it as $\mathbf{y}=\left(
\begin{array}{l}
{\mathbf{y}_1}\\
{\mathbf{y}_2}
\end{array}
\right)$, where $\mathbf{y}_i\in R^{N-1}$, $i=1,2$. We suppose $\mathbf{y}\neq 0$.
If $\mathbf{y}_2\neq 0$, we can have
\begin{equation}
\mathbf{y}^T \mathbf{M}_{\delta} \mathbf{y}=  \mathbf{y}^T \mathbf{M}_{0} \mathbf{y}+\delta \mathbf y_2^T\mathbf{y}_2
\geq \delta \mathbf y_2^T\mathbf{y}_2>0.
\end{equation}
where $ \mathbf{M}_{0} = \left(\begin{array}{cc}
\mathbf{A} &\mathbf{B} \\
\mathbf{B} &\mathbf{C}
\end{array}\right)$ and we have used the fact that 
$$\mathbf{y}^T \mathbf{M}_{0} \mathbf{y}=\Phi_h(u_1,\cdots,u_{N-1},x_1,\cdots,x_{N-1};\mathbf{y}_1,\mathbf{y}_2)$$
 is semi-positive definite by its definition~\eqref{e:dissip_h}. 
Otherwise, if $\mathbf{y}_2=0$, then $\mathbf{y}_1\neq 0$. We have
\begin{equation}
\mathbf{y}^T \mathbf{M}_{\delta} \mathbf{y}=  \mathbf{y}_1^T \mathbf{A} \mathbf{y}_1>0,
\end{equation}
for any given  partition $\mathcal{T}(X(t))$ of the interval $I$, since the matrix $\mathbf{A}$ is
the standard mass matrix of the linear finite element method with respect to the partition $\mathcal{T}(X(t))$
and thus positive definite.
Combine the analysis together, we show that  $\mathbf{M}_{\delta}$ is positive definite.
\end{proof}

By this lemma,  the ODE system \eqref{e:ODE} can be rewritten as
\begin{equation}\label{e:ODE1}
\left(
\begin{array}{l}
\dot{\mathbf{u}}\\
\dot{\mathbf{x}}
\end{array}
\right)=
\mathbf{M}_{\delta}^{-1}
\left(
\begin{array}{l}
{\mathbf{f}}\\
{\mathbf{g}}
\end{array}
\right).
\end{equation}
The  equation has a unique solution  for any initial value $\mathbf{u}(0)$ and $\mathbf{x}(0)$ when the right hand side 
function $\left(
\begin{array}{l}
{\mathbf{f}(\mathbf{u},\mathbf{x})}\\
{\mathbf{g}(\mathbf{u},\mathbf{x})}
\end{array}
\right)$ is Lipschitz continuous with respect to $\mathbf{u}$ and $\mathbf{x}$.
The stiffness of the
  system and its numerical solution has been analyzed in \cite{miller1981movingB,wathen1985structure}. 
When the numerical solution is non-degenerate in $V_h^N$, the stabilization parameter $\delta$ can be chosen as zero.


\section{Numerical analysis}
One important advantage of nonlinear approximations is that they have better  accuracy
than linear approximations. This is illustrated in Lemma~\ref{lem:approx} for the interpolation error.
One would expect that the MFEM also has better accuracy than the standard FEM.
However, the theoretical analysis is very difficult. 
In this section, we consider only the stationary solution when $t$ goes to infinity.

\subsection{The energy decay property}
We first prove the following discrete energy decay property of the ordinary differential system~\eqref{e:ODE}. 
\begin{theorem}\label{th:decay}
Let $(\mathbf{u}(t),\mathbf{x}(t))$ be 
the solution of the equation \eqref{e:ODE} and $u_h(t,x)\in V_h^N$ be the corresponding piecewise linear function.
Then we have
\begin{equation}
\frac{d \E(u_h)}{d t}
\leq 0,
\end{equation}
where the equality holds only when $\dot{x}_k=0$ and $\dot{u}_k=0$ for $k=1,\cdots,N-1$.
\end{theorem}
\begin{proof}
Notice that 
$$\frac{d \E(u_h)}{d t}=\sum_{k=1}^{N-1}(\frac{\partial \E_h }{ \partial u_k}\dot{u}_k+\frac{\partial \E_h }{\partial x_k}\dot{x}_k)=
-\xi \sum_{k=1}^{N-1}(f_k\dot{u}_k+g_k\dot{x}_k)=-\xi (\dot{\mathbf{u}}^T,\dot{\mathbf{x}}^T) \mathbf M_{\delta} \left(
\begin{array}{l}
\dot{\mathbf{u}}\\
\dot{\mathbf{x}}
\end{array}
\right),
$$
where we have used the equation \eqref{e:ODE}. 
By Lemma~\ref{lem:posdef}, we finish the proof of the theorem.
\end{proof}

From the proof, we easily see that 
\begin{equation}\label{e:disStru}
\frac{d \E(u_h)}{d t}=-2\Phi_h^{\delta}.
\end{equation}
This implies the semi-discrete problem~\eqref{e:ODE} preserves the
energy dissipation structure \eqref{e:engdecay1} of the original gradient flow system. 



\subsection{Error analysis for the equilibrium state}
We now consider the stationary solution of \eqref{e:modelpb1}.
When $t$ goes to infinity, we suppose the solution $u(t,x)$ of \eqref{e:modelpb1} converge to a function $u^\infty(x)$
satisfying
\begin{equation}\label{e:EL_S}
\int_I \alpha \partial_x u^\infty\partial_x v+ f(x,u^\infty) v \dx=0,\qquad \forall v\in H^1_0(I).
\end{equation}
 To do analyse for this problem, we need one more assumption for the Lipschitz constant $L_0$ in \eqref{e:Lip} that
\begin{itemize} 
\item[(H1)]{\qquad\qquad\qquad\qquad ${\alpha}- L_0 c_0^2 >c_1$, \qquad\ \  for some constant $c_1>0$,}
\end{itemize}
where $c_0$ is a constant in the Poincare inequality
\begin{equation}\label{e:Sob}
\|u\|_{L^2(I)}\leq c_0|u|_{H^1(I)}.
\end{equation}
Under this assumption, it is easy to see that the equation~\eqref{e:EL_S} is elliptic and has a unique solution.
Furthermore, $u^{\infty}$ is the unique minimizer for the energy minimization problem
\begin{equation}\label{e:conEng}
\inf_{v\in H^1_0} \E(v).
\end{equation} 


In the following, we consider the approximation of the problem~\eqref{e:EL_S} by the MFEM.
For the stationary solution for  the discrete problem~$\eqref{e:ODE}$, we have the following theorem.
\begin{theorem}\label{th:engRel}
Let $(\mathbf{u}(t),\mathbf{x}(t))$ be 
the solution of the equation \eqref{e:ODE}  which corresponds to a nondegenerate function $u_h(t,x)\in V_h^N$.
 When $t$ goes to infinity,  $(\mathbf{u}(t),\mathbf{x}(t))$ will converge to a vector $(\mathbf{u}^{\infty},\mathbf{x}^{\infty})$.
 Suppose that $\mathbf{x}^{\infty}$  corresponds to a partition of $I$, then
$u_h(t,x)$ converges to a function $u_h^{\infty}(x)\in V_h^N$ which satisfies 
\begin{equation}\label{e:discreteEL}
\left\{
\begin{array}{l}
\int_{I}\alpha\partial_x u_h^{\infty} \partial_x \phi_k+f(x,u_h^{\infty})\phi_k dx=0, \qquad\qquad\qquad\qquad\quad\  k=1,\cdots, N-1,\\
\int_{I}\alpha\partial_x u_h^{\infty} \partial_x \beta_k+f(x,u_h^{\infty})\beta_k dx-\frac{\alpha}{2}[(\partial_x u_h^{\infty} )^2]|_{x_{k}^{\infty}}=0, \qquad k=1,\cdots, N-1.
\end{array}
\right.
\end{equation}
Furthermore, if $u_h^{\infty}$ is nondegenerate and asymptotic stable, it is a local minimizer of  $\E(u_h)$ in $V_h^N$.
\end{theorem}

\begin{proof}
By Equation~\eqref{e:disStru}, we have the energy decay property
$$
\frac{d \E(u_h(t))}{dt}=-2\Phi(\partial_t u_h)-\delta\xi \sum_{k=1}^N\dot{x}_k^2\leq 0, \qquad\qquad \forall t>0.
$$
Therefore, we have
\begin{align}\label{e:temp1}
\E(u_h(T))=\E(u_h(0))-\xi\int_0^T\int_I(\partial_t u_h)^2dx dt-\delta\xi \sum_{k=1}^N\int_0^T\dot{x}_k^2dt, \qquad \hbox{for } T>0.
\end{align}
Notice that
$$
\E(u_h(T))\geq \inf_{v_h\in V_h^N}\E(v_h)\geq\inf_{v\in H^1_0(I)}\E(v)=\E(u^{\infty}).
$$
Thus $\E(u_h(T))$ is bounded from below when $T$ goes to infinity.
By Equation~\eqref{e:temp1}, we have $\lim_{t\rightarrow \infty}\dot{x}_k(t)=0$ and $\lim_{t\rightarrow \infty}\int_I(\partial_t u_h)^2dx=0$.
Since $x_k$ is in the bounded interval $I$, $x_k(t)$ converges when $t$ goes to infinity.
Similarly, $u_k(t)$ also converges when $t$ goes to infinity.
When the limit $\mathbf{x}^{\infty}$ generates a partition of $I$, 
then $(\mathbf{u}^{\infty},\mathbf{x}^{\infty})$ corresponds to a function $u_h^{\infty}$ in $V_h^N$.
From Equation~\eqref{e:ODE}, we have 
\begin{equation*}
\lim_{t\rightarrow\infty} f_k(\mathbf{u},\mathbf{x}) =0, \quad \lim_{t\rightarrow\infty} g_k(\mathbf{u},\mathbf{x})=0, \qquad \hbox{for } k=1,\cdots,N-1.
\end{equation*}
This leads to the equation~\eqref{e:discreteEL}, i.e.
$u_h^\infty$ is a critical point of the  energy $\E(\cdot)$ in the manifold $V_h^N$.

When $u_h^{\infty}$ is nondegenerate and asymptotic stable,  it is a local minimizer of $\E$ in $V_h^N$. 
Otherwise, if it is a critical point but not a local minimizer of $\E$, there exists a non-trivial trajectory near $u_h^{\infty}$ which makes
the energy nondecreasing. This contradicts with the energy decay property in Theorem~\ref{th:decay}.
\end{proof}

\begin{remark}
If $\E(u)$ has a quadratic form, Equation~\eqref{e:EL_S} is a linear equation.  Then the theorem implies that $u_h^{\infty}$ is
a locally best approximation in energy norm. For example, suppose
$$\E(u)=\int_{I}\frac{\alpha}{2}(\partial_x u)^2+f(x) u \dx,$$
then by  Equation~\eqref{e:EL_S}, we have
\begin{align*}
\E(v_h)&= \int_{I}\frac{\alpha}{2}(\partial_x v_h)^2+f(x) v_h \dx\\
&=\int_{I}\frac{\alpha}{2}(\partial_x v_h)^2- \alpha \partial_x u^\infty \partial_x v_h \dx\\
&=\int_{I}\frac{\alpha}{2}(\partial_x v_h-\partial_x u^\infty )^2dx -\int_{I}\frac{\alpha}{2} (\partial_x u^\infty)^2 \dx.
\end{align*}
Therefore, Theorem~\ref{th:engRel} implies that $u_h^\infty$ is a locally best approximation of $u^\infty$ in the energy norm
when $u_h^{\infty}$ is nondegenerate and asymptotic stable.
This is one main result obtained in \cite{jimack1996best}.
\end{remark}


Notice the admissible set $V_h^N$ for the discrete function is a $2(N-1)$ dimensional manifold in $H^1_0(I)$.
In general the energy minimization problem $\E(\cdot)$ in the manifold $V_h^N$ is not convex, even when $\E(\cdot)$ is 
a convex functional. This is very different from the standard FEM, noticing that the minimizer of 
a convex functional in a linear space is unique. 
Therefore, the local best approximations given in Theorem~\ref{th:engRel} is optimal. 
However, when the initial value
is in the domain of attraction of a global minimizer of $\E(\cdot)$ in $V_h^N$, one can approximate the global minimizer 
by the MFEM.
The following theorem shows that the global minimizer has an optimal convergence order.

\begin{theorem}\label{theo:main}
Let $u_h^{\infty}$  be a global minimizer of $\E(\cdot)$ in $V_h^N$. Assume that $f(x,u)$ is differentiable with respect to $u$.
 If $u^\infty \in B^{s+1}_q(L_q(I))$ with $0\leq s\leq 1$ and $1/q<s+1/2$,
then we have
%
\begin{equation}
 |u^{\infty}-u_h^{\infty}|_{H^1}\lesssim N^{-s} |u^\infty|_{ B^{s+1}_q(L_q(I))}.
 \end{equation}
\end{theorem}
\begin{proof}
Denote $w= u_h^{\infty}-u^{\infty}$. Since $u_h^\infty\in V_h^N\subset H^1_0(I)$ and $u^{\infty}$ is the minimizer of $E(u)$ in $H^1_0(I)$,  then we have
$w \in H^1_0(I)$ and
\begin{align}
&\E(u_h^{\infty})-\E(u^{\infty}) \nonumber\\
&=\int_{I}\Big(\frac{\alpha}{2}(\partial_x u_h^\infty )^2-\frac{\alpha}{2}(\partial_x u^\infty)^2\Big)+\Big(F(x,u_h^{\infty})  
-  F(x,u^{\infty})\Big)\dx\nonumber\\
&=\int_{I}\frac{\alpha}{2}(\partial_x w)^2+\alpha\partial_x u^\infty \partial_x w+
 f(x,u^{\infty}) w +\frac{1}{2}\partial_{u}f(x,u^\infty+s w)w^2 \dx\nonumber\\
&=\int_{I}[\alpha\partial_x u^\infty \partial_x w+
 f(x,u^\infty) w ]+\frac{\alpha}{2}(\partial_x w)^2+\frac{1}{2}\partial_{u}f(x,u^\infty+s w)w^2 \dx\nonumber\\
&=\int_{I}\frac{\alpha}{2}(\partial_x w)^2+\frac{1}{2}\partial_{u}f(x,u^\infty+s w)w^2 \dx,\label{e:temp2}
\end{align}
where $0\leq s\leq 1$ and we have used the equation~\eqref{e:EL_S} in the last equation. From the Lipschitz condition~\eqref{e:Lip} of $f$,
we know that $|\partial_{u}f(x, u)|\leq L_0$. This further leads to
\begin{align}
\E(u_h^{\infty})-\E(u^{\infty})&\geq \int_{I}\frac{\alpha}{2}(\partial_x w)^2\dx -\frac{L_0}{2}\int_I w^2 \dx \nonumber \\
&\geq \frac{\alpha-L_0 c_0^2}{2} \int_I (\partial_x w)^2\dx\nonumber \\
&\geq \frac{c_1}{2} \int_I (\partial_x w)^2\dx=\frac{c_1}{2}|\partial_x  u^{\infty}-\partial_x u^{\infty}_h |^2_{H^1}.\label{e:temp}
\end{align}
Here we have used  assumption (H1) and the Poincare inequality~\eqref{e:Sob}.

For any $v_h\in V_h^N$,  we denote $\tilde{w}=v_h-u^\infty$. Noticing that $u_h^{\infty}$ is a global minimizer 
of $\E$, similar calculations as \eqref{e:temp2} gives 
\begin{align*}
\E(u_h^{\infty})-\E(u^{\infty})&\leq  \E(v_h)-\E(u^{\infty}) \\
&=\int_{I}\frac{\alpha}{2}(\partial_x \tilde{w})^2+\frac{1}{2}\partial_{u}f(x,u^\infty+s\tilde{w})\tilde{w}^2 \dx\\
&\lesssim  \|u^{\infty}-v_h\|_{H^1}^2,\qquad \forall v_h\in V_h^N.
\end{align*}
Combine  the inequality with \eqref{e:temp} and using the Poincare inequality again, we obtain
\begin{equation}
 |u^{\infty}-u_h^{\infty}|_{H^1}^2\lesssim \E(u_h^{\infty})-\E(u^{\infty})\lesssim \inf_{v_h\in V_h^N}\|u^{\infty}-v_h\|_{H^1}^2 \lesssim (\sigma_N(u))^2.
\end{equation}
By Lemma~\ref{lem:approx}, we have finished the proof.


\end{proof}
In general it is not clear how to determine the domain of  attraction of a global minimizer.
There is no guarantee that  all initial values can tend to a  globally stationary solution.
As suggested by the numerical results in \cite{jimack1996best}, it is possible that the 
solution of a MFEM does not go to the global minimizer when some elements in the partition
tend to degenerate. 

\section{Numerical Examples}
As we mentioned in Section 3, the equation~\eqref{e:ODE} is a stiff ODE system.
One reason is that some mesh size in the MFEM can be very small. 
In \cite{miller1981movingB}, a technique by adding ``internodal spring force" was
 developed to keep the grid points at least slightly separated.  
Here we use a  similar technique by adding a penalty term to the total energy. 
We replace $\E(v_h)$ by 
\begin{equation}\label{e:modifyEng}
\E_{\tilde{\delta}} (v_h)=\E(v_h)+\frac{\tilde{\delta}}{N}\sum_{k=0}^N(\ln(N h_k/L))^2.
\end{equation}
where $h_k=x_{k+1}-x_k$ and $\tilde{\delta}$ is a small parameter. 
The penalty term goes to infinity when $h_k$ goes to zero. 

In the following, we show some numerical examples. In our simulations, we solve the problem \eqref{e:ODE} by a simple 
forward Euler scheme with very small time step. More efficient
solvers(e.g. some implicit stiff ODE solver) have been studied in \cite{miller1981movingB}. 
\subsection{A linear diffusion equation}
We first consider a linear diffusion problem:
\begin{equation}\label{e:lineardiffusion}
\left\{
\begin{array}{ll}
 \partial_t u-\partial_{xx} u=f(x), & x\in (a,b),t>0;\\
 u(x,0)=u_0(x),& x\in (a,b);\\
 u(a,t)=u(b,t)=0 & t>0.
\end{array}
\right.
\end{equation}
This corresponding energy density is given by
\begin{equation}\label{e:energydiffusion}
\E(u)=\int_{I}\frac{1}{2}|\partial_x u|^2 -f(x)u(x) \dx,
\end{equation}
where $I=(a,b)$.
We choose $\delta=0.0001$ and $\tilde{\delta}=0.01$. We test for  several different $u_0$ and $f$. 


{\bf Example 1.} In the first example, we set $I=(0,1)$ and $f(x)=\delta_{a_0}(x)$ 
which is a Dirac measure at the middle point $a_0=0.5$. 
We have $\int_{0}^1 f(x)v(x)\dx=v(x_0)$, for any $v\in C_0^1([0,1])$.
The initial solution is given by $u_0= \sin(\pi x)+x\chi_{\{x<x_0\}}+(1-x)\chi_{\{x>x_0\}}$.
In this case the equation~\eqref{e:lineardiffusion} has an analytic solution 
$$u(x,t)=\sin(\pi x)e^{-\pi^2 t}+x\chi_{\{x<x_0\}}+(1-x)\chi_{\{x>x_0\}}.$$
The spacial derivative of the  solution has a jump  at the middle point $a_0=0.5$. 

The initial partition  of $I$ is uniform and given by 
$0=x_0<\frac{1}{N}<\cdots<\frac{N-1}{N}<x_N=1$. 
Here we choose $N=5,9,19,39,79$ so that the non-smooth point $a_0$ is
not a node in the initial partition.
 We solve the problem \eqref{e:ODE} until $T=0.04$. 
 The $H^1$ error and the $L^2$ error between the discrete solution $u_h$ and the exact solution $u$ at $T$
 are computed by
\begin{align*}
&err_{H^1}:=\Big(\int_I(\partial_x u(x,T)-\partial_x u_h(x,T))^2 \dx\Big)^{1/2},\\
 &err_{L^2}:=\Big(\int_I( u(x,T)- u_h(x,T))^2 \dx\Big)^{1/2}.
\end{align*}

\begin{table}[h]\small
\caption{\small The $H^1$-norm and $L^2$-norm of the error in Experiment~1.}\label{tab:ErrorEx1}
\vspace{-0.2cm}
\begin{center}
\begin{tabular}{l|lc|lc}
\hline
{\bf Adaptive} & $err_{H^1}$  & order &  $err_{L^2}$   & order \\
\hline
$N=5$&0.3326 & --  &0.0338 &-- \\
$N=9$ &0.1571& 1.28 & 0.0098&2.11\\
$N=19$ &0.0842 &0.83& 0.0022 &2.00 \\
$N=39$&0.0470& 0.81&0.000523&1.99 \\
$N=79$&0.0265& 0.81&0.000141& 1.86 \\
\hline
\end{tabular}\\
\vspace{0.2cm}
\begin{tabular}{l|lc|lc}
\hline
{\bf Uniform} &  $err_{H^1}$  & order &  $err_{L^2}$   & order  \\
\hline
$N=5$&0.6363 & --  &0.0435 &-- \\
$N=9$ &0.1571& 0.72 & 0.0153&1.77 \\
$N=19$ &0.2587 &0.64& 0.0043 &1.70 \\
$N=39$&0.1703& 0.58&0.0013&1.66 \\
$N=79$&0.1159& 0.54&0.000438& 1.54 \\
\hline
\end{tabular}
\end{center}
\end{table}
In Table~\ref{tab:ErrorEx1}, we show the numerical errors for the different choice of $N$. The convergence order is computed 
by $s_i:=\ln(err_i/err_{i+1})/\ln(N_{i+1}/N_{i})$, which implies that 
the errors decrease with order $O(N^{-s})$. We can see that the $H^1$-error is almost of order $O(N^{-0.8})$, close to the optimal
convergence rate $O(N^{-1})$. Meanwhile, the $L^2$-error is of order $O(N^{-2})$, which is  optimal in  
approximation theory. In comparison, we also show the errors computed by the standard finite element method
 on a uniform partitions of $I$. We could see that the $H^1$ norm is almost of order $O(N^{-0.5})$ and
the $L^2$-error is of order $O(N^{-1.5})$. This is reasonable since the exact solution is not smooth.
In this case, the MFEM has better accuracy than the standard FEM. 

In Figure~\ref{fig:Exampl1} we illustrate the numerical solutions for the case $N=9$ at different time $t$.
We can see that initially the interval $I$ is uniformly divided into $N-1$ cells and $u_k$ are simply the interpolation of $u(0,x)$.
There is a large error at $a_0=0.5$ since the initial function is not smooth there. With time increasing,
we can see that the  notes $x_k$ move with time. In particular, the two notes near the non-smooth point $a_0$ approach
to the point gradually. The distance between them become smaller and smaller. At about $t=0.0048$, the right point almost arrives at
$a_0$. After that, the left point moves away from the point to further decrease the computational error.
 \begin{figure}[ht!]
 \centering
  \subfigure[$t=0$]{
   \includegraphics[width=0.31\textwidth]{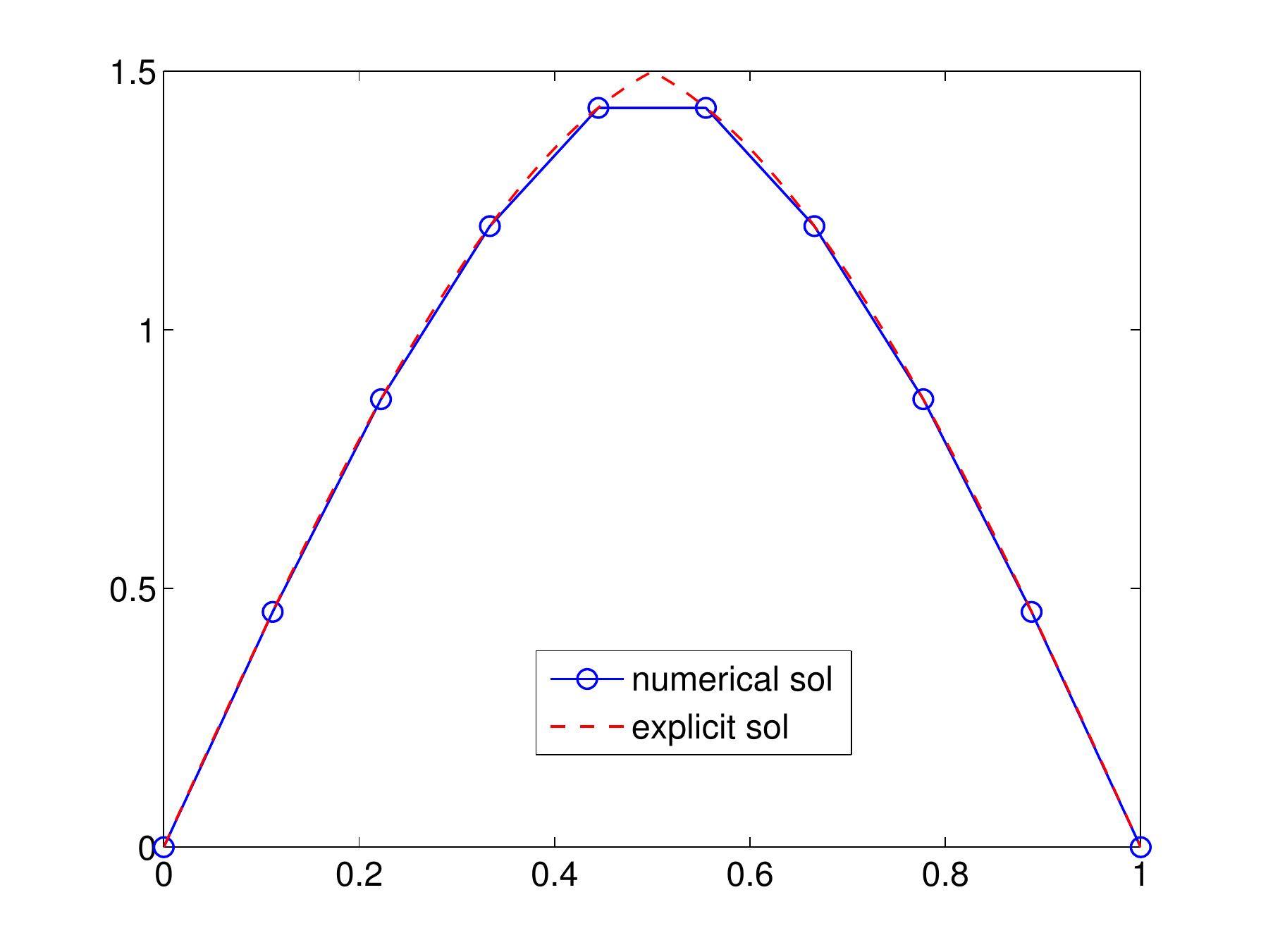}
   }
  \subfigure[$t=0.0012$]{
   \includegraphics[width=0.31\textwidth]{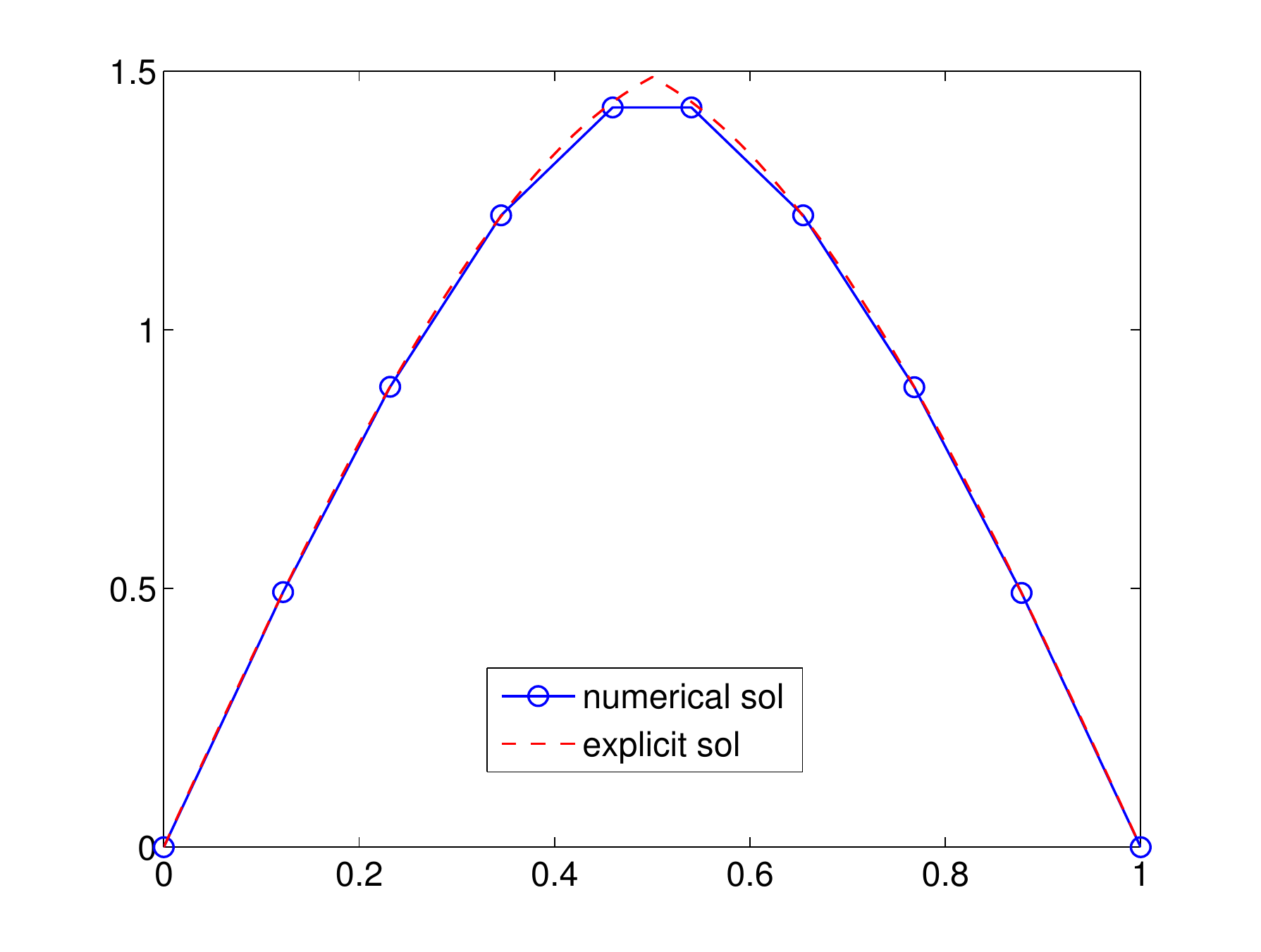}
}
  \subfigure[$t=0.0024$]{
   \includegraphics[width=0.31\textwidth]{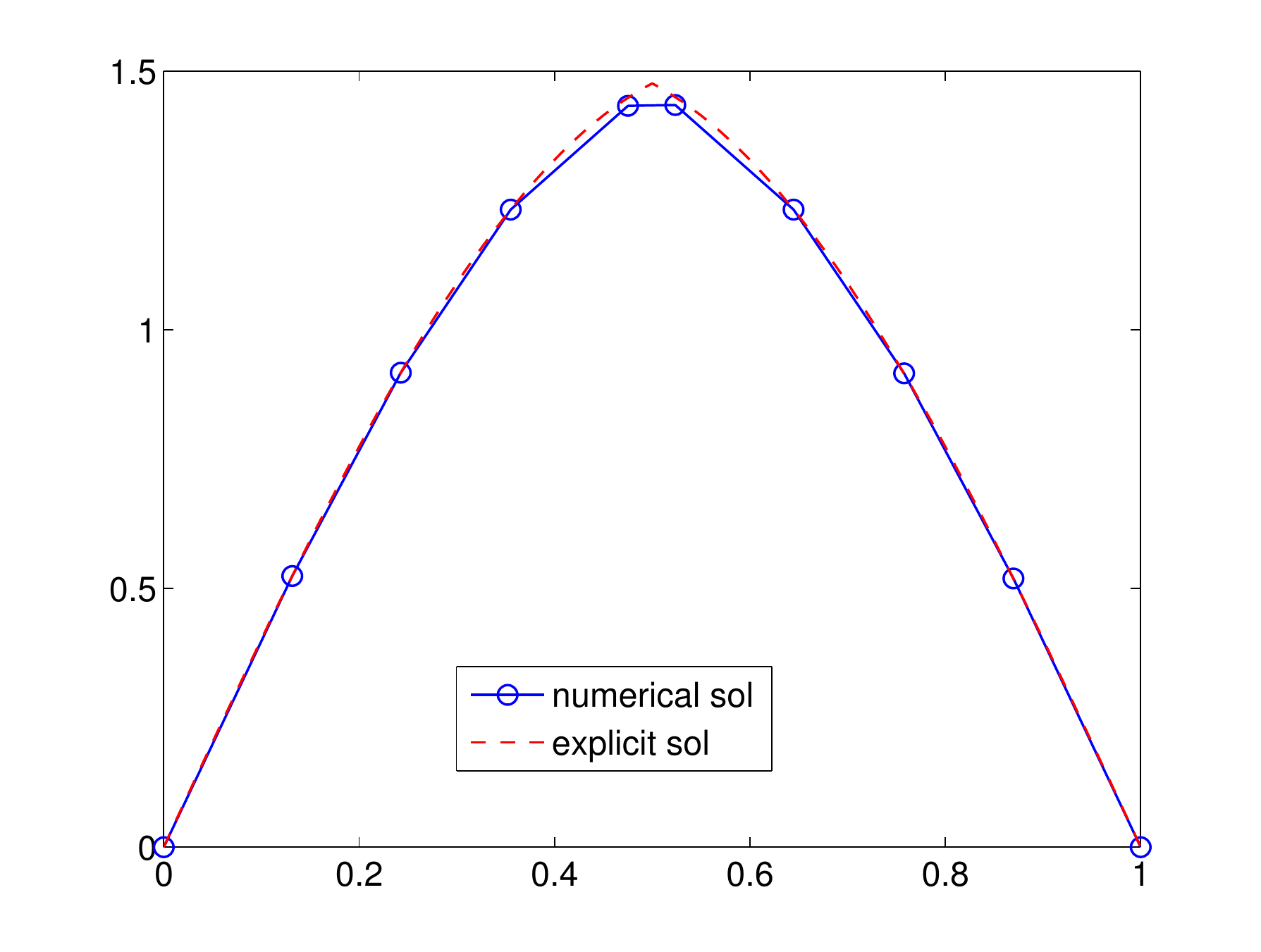}
}
  \subfigure[$t=0.0036$]{
   \includegraphics[width=0.31\textwidth]{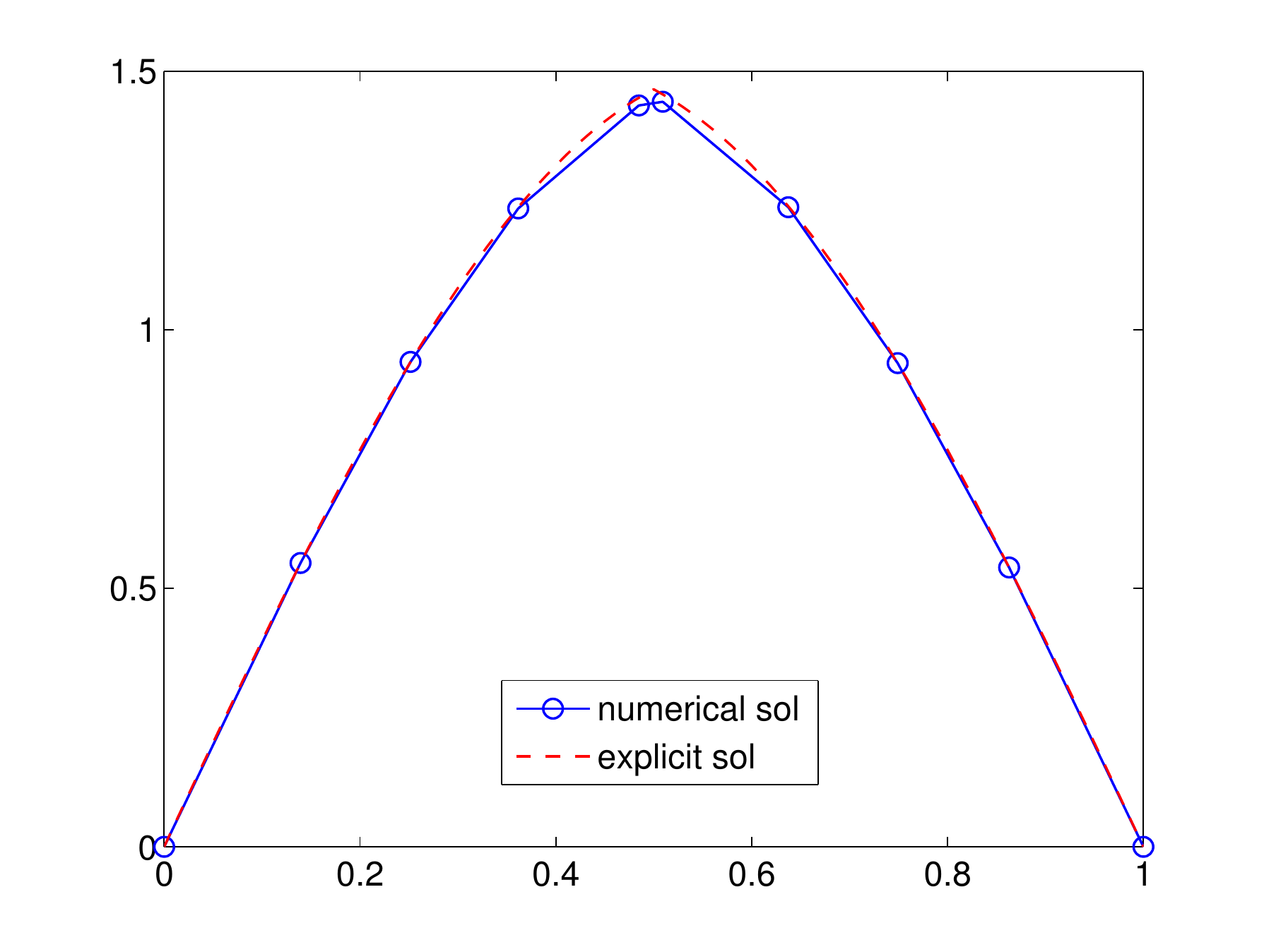}
}
  \subfigure[$t=0.0048$]{
   \includegraphics[width=0.31\textwidth]{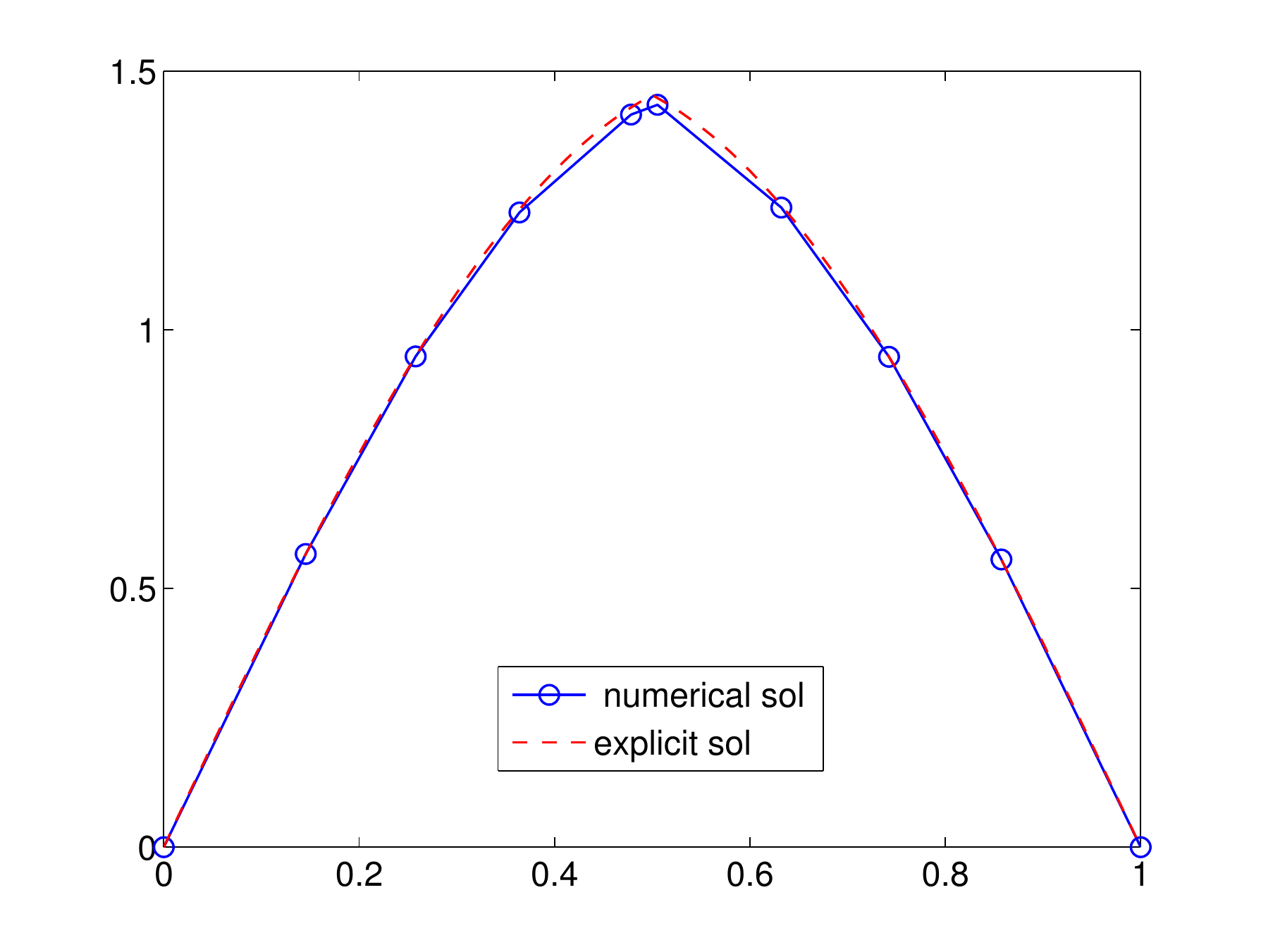}
}
  \subfigure[$t=0.006$]{
   \includegraphics[width=0.31\textwidth]{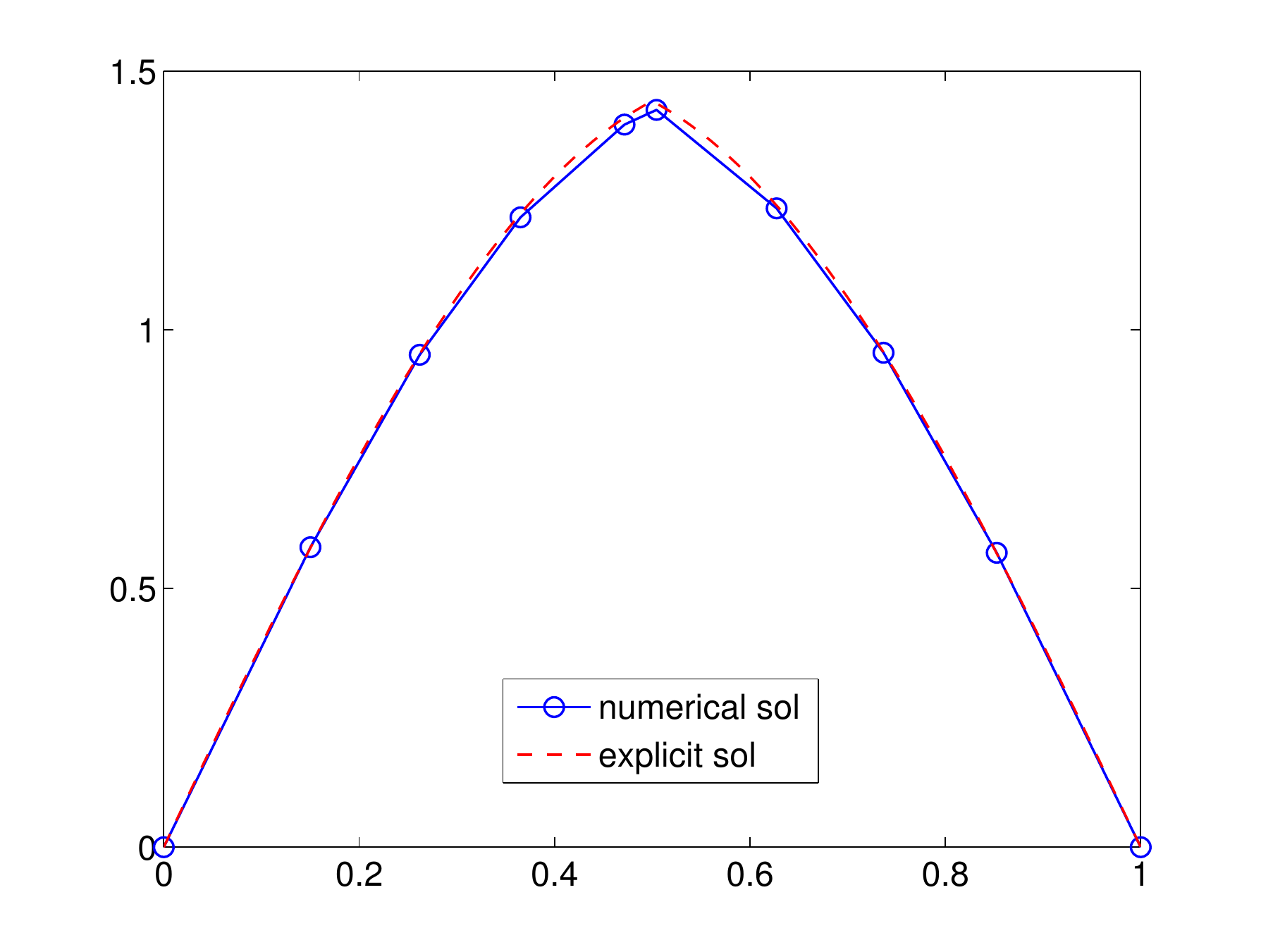}
}
  \subfigure[$t=0.009$]{
   \includegraphics[width=0.31\textwidth]{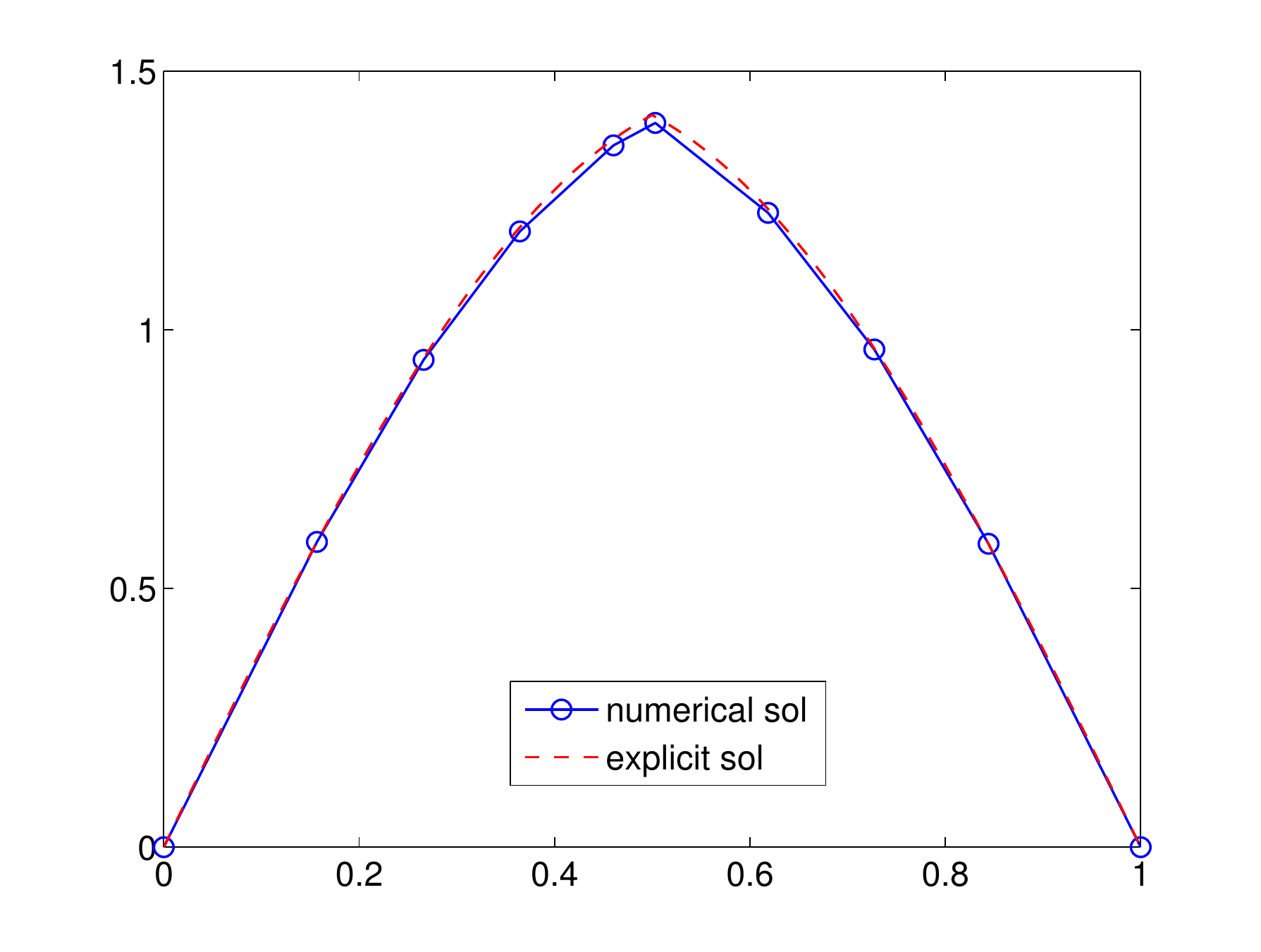}
}
  \subfigure[$t=0.014$]{
   \includegraphics[width=0.31\textwidth]{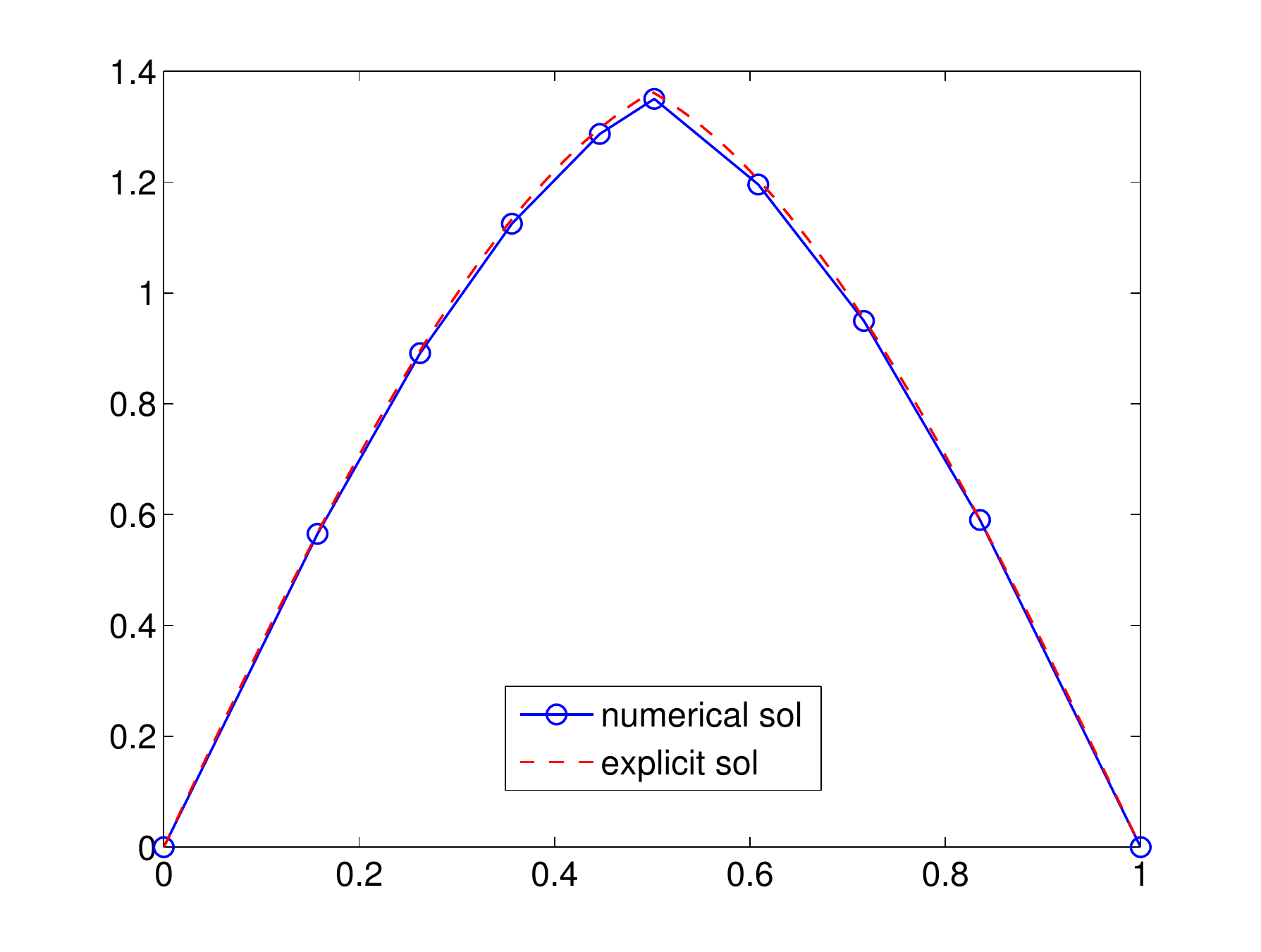}
}
  \subfigure[$t=0.04$]{
   \includegraphics[width=0.31\textwidth]{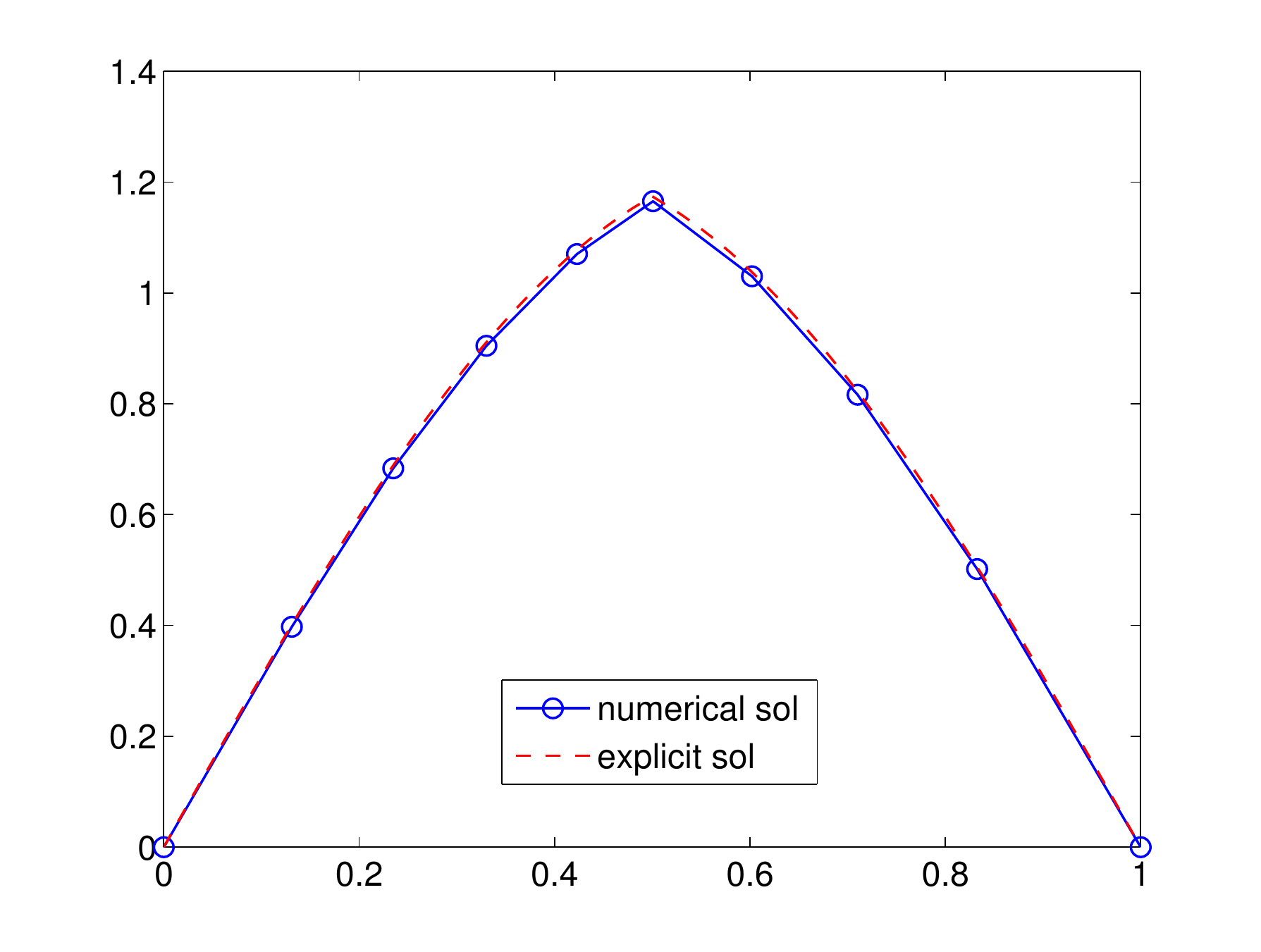}
}
 \caption{The profile of the numerical solution and the exact solution at different time $t$ when $N=9$ for Example 1.
 The small circles in the figures show the computational solutions $u_k$ corresponding to $x_k$, for $k=0,\cdots,9$.
 }
 \label{fig:Exampl1}
\end{figure}

{\bf Example 2.} To further illustrate the behaviour of the MFEM, we consider a diffusion problem with 
an almost singular initial condition. We set $f=0$ and $I=(-3,3)$.
 We choose $u_0=\frac{1}{\sqrt{0.004\pi}}(e^{-x^2/(0.004)}-e^{-9/(0.004)})$, which is a very sharp 
Gaussian function as shown in the first sub-figure in Figure~\ref{fig:Example2}.  $u_0$ is almost equal to
zero outside of a narrow interval. For this problem, we do not have an explicit formula for the analytic solution.
 
Initially we choose a grid whose points concentrated in a small interval $(-0.2,0.2)$ to resolve the profile of $u_0$:
\begin{equation*}
X(0)=\{-3,-0.2,\cdots,-0.2+i\frac{0.4}{N-2},\cdots,0.2,3\}.
\end{equation*}

In Figure~\ref{fig:Example2}, we show the numerical solution when $N=19$. 
We can see that the numerical solutions spread outwards gradually while 
the height of the profile decreases.
In particular, the grid points gradually adjust their positions 
so that  the Gaussian profile of the numerical solution is always well resolved.
 \begin{figure}[ht!]
 \centering
  \subfigure[$t=0$]{
   \includegraphics[width=0.45\textwidth]{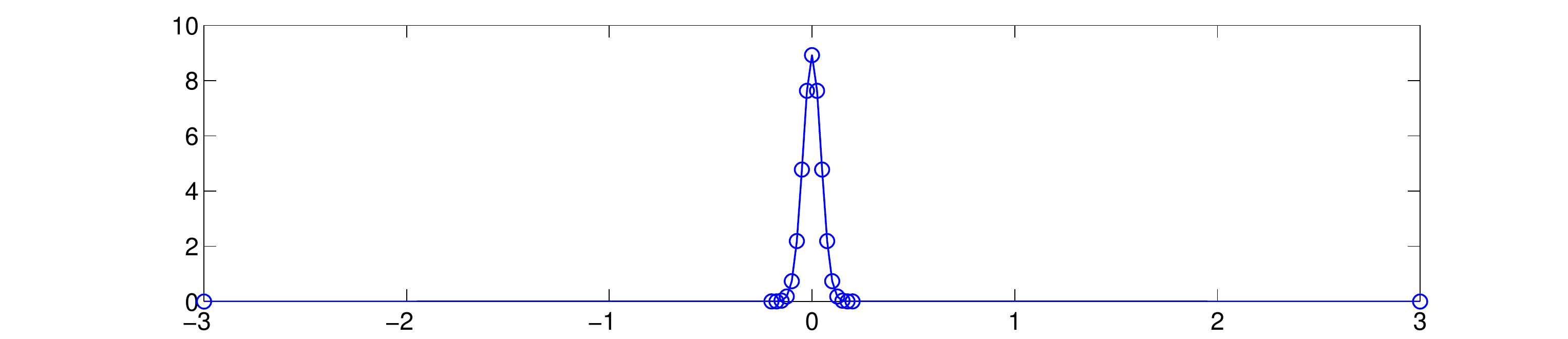}
   }
  \subfigure[$t=0.001$]{
   \includegraphics[width=0.45\textwidth]{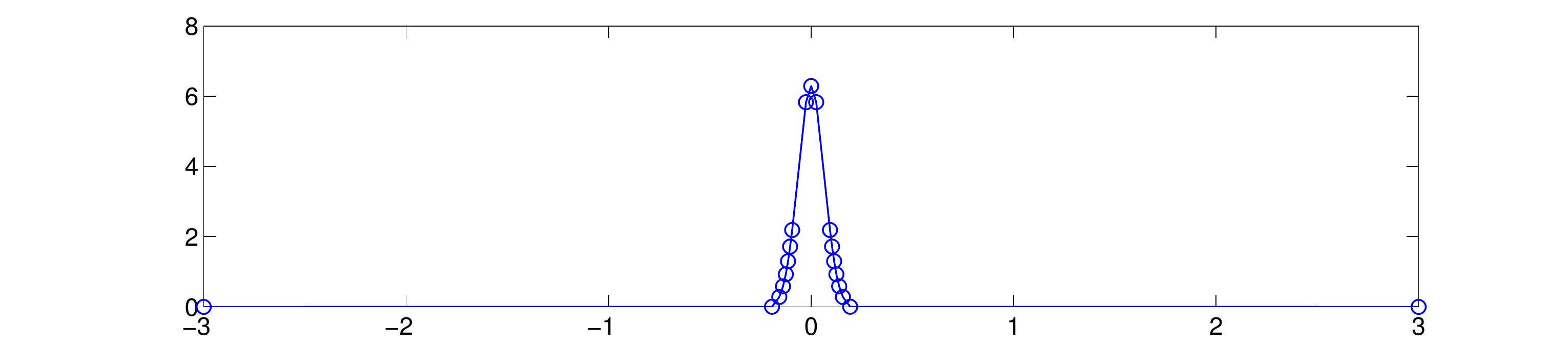}
}
  \subfigure[$t=0.005$]{
   \includegraphics[width=0.45\textwidth]{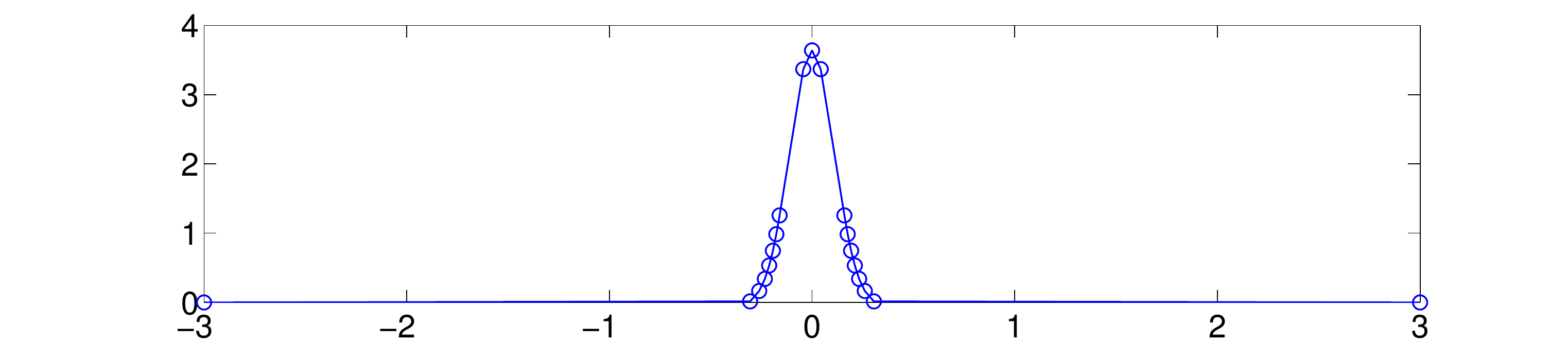}
}
  \subfigure[$t=0.02$]{
   \includegraphics[width=0.45\textwidth]{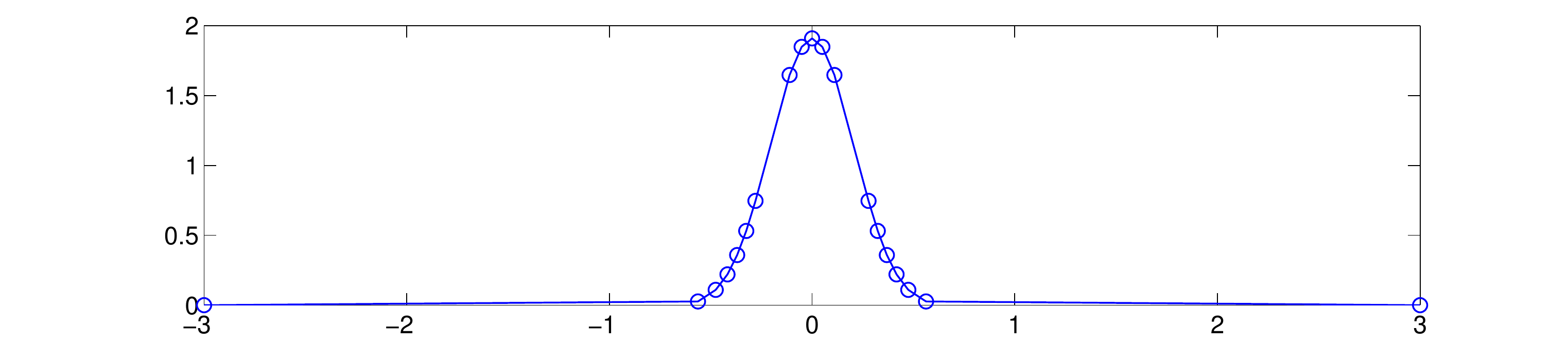}
}
  \subfigure[$t=0.1$]{
   \includegraphics[width=0.45\textwidth]{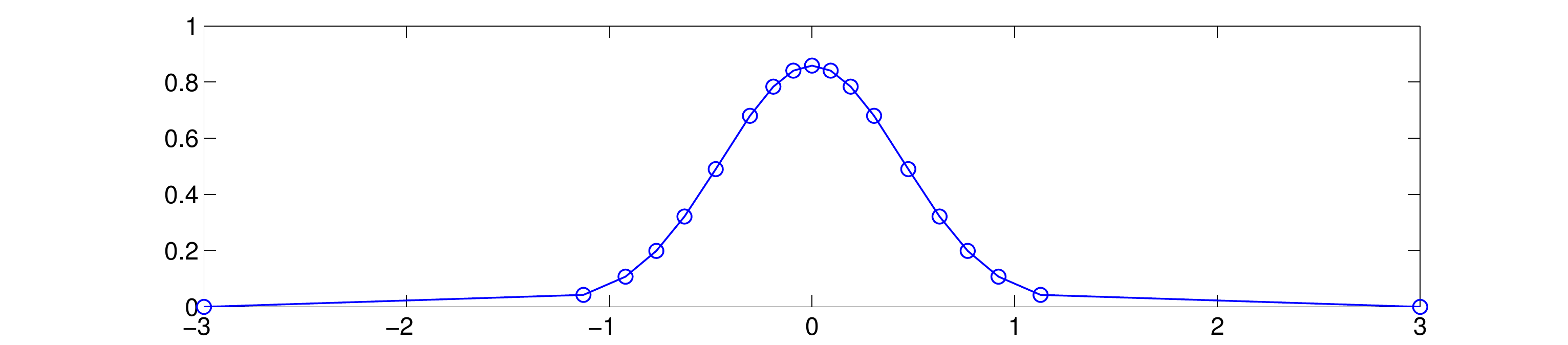}
}
  \subfigure[$t=0.2$]{
   \includegraphics[width=0.45\textwidth]{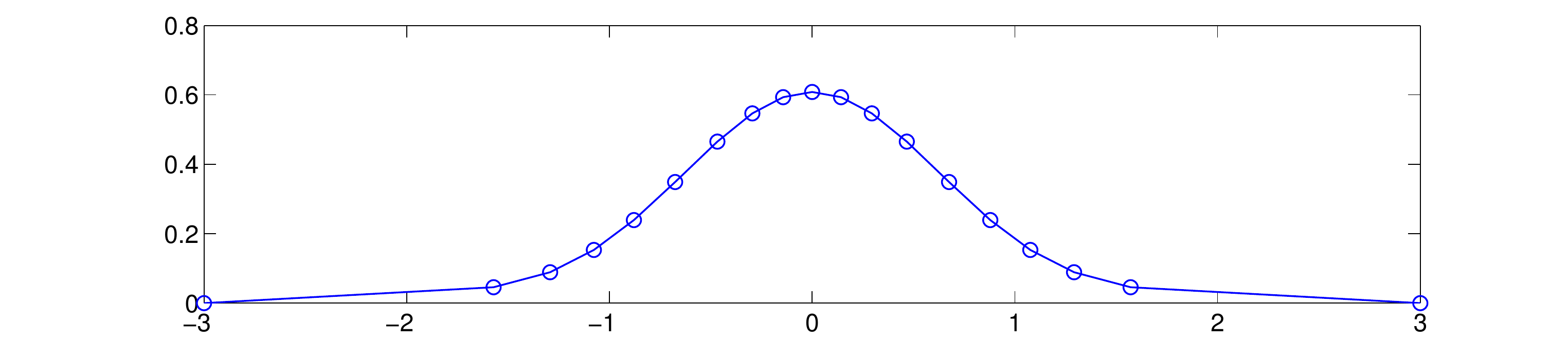}
}

 \caption{The profile of the numerical solution at different time $t$ when $N=19$ for Example 2.
 The small circles in the figures show the computational solutions $u_k$ corresponding to $x_k$, for $k=0,\cdots,19$.}
 \label{fig:Example2}
\end{figure}
\subsection{The Allen-Cahn equation}
We  then consider an Allen-Cahn equation in a interval $I=(0,1)$,
\begin{equation}\label{e:AC}
\left\{
\begin{array}{ll}
\partial_t \phi -\eps\partial_{xx} \phi +\frac{\phi^3-\phi}{\eps}=0,& x\in (0,1),\\
\phi(x,0)=x, & x\in (0,1),\\
\phi(0,t)=0, \phi(1,t)=1,
\end{array}
\right.
\end{equation}
 which is a nonlinear example for the model problem~\eqref{e:modelpb1}.
The corresponding free energy is 
\begin{equation}
\E(\phi)=\int_0^1\frac{\eps}{2}(\partial_x \phi)^2 +\frac{(1-\phi^2)^2}{4\eps} \dx.
\end{equation}
We are interested in the stationary profile of the problem. When $\eps$ is small enough, the  profile is approximated nicely by
\begin{equation}
\phi^{\infty}(x)=\tanh(\frac{x-0.5}{\sqrt{2}\eps}).
\end{equation} 
It has a very thin inner layer at the middle point $x=0.5$. 
In simulations, we choose $\delta=0.0001$, $\tilde{\delta}=0.0001$ and test for  different values of $\eps$.  
The initial condition is given by 
$$\phi(0,x)=2(x-0.5).$$

{\bf Example 3.} In this example, we set $\eps=0.05$ and test the method for different choice of $N$. 
Initially, we give a uniform partition of $I$ that
 $X(0)=\{0=x_0<\frac{1}{N}<\cdots<\frac{N-1}{N}<x_N=1\}$,
We solve the problem for various choices of $N=5,10,20,40,80$. 

In the computation, the profile of $\phi_h$ changes gradually to a stationary state $\phi_h^\infty$.
Meanwhile the grid points move
into a highly non-uniform distribution to resolve the inner layers of the solution.
We illustrate some plots of the final solution $\phi_h^{\infty}$ as well as the distributions of $x_k$ in Figure~\ref{fig:Example3}.
They are obtained by solving the problem \eqref{e:ODE} until the decrease of the discrete energy
 in one step is smaller than a tolerance $TOL=1e^{-10}$. 
The profile of the ``exact" solution $\phi^{\infty}$ is also shown.
 We can see that the numerical solution fits the exact solution very well 
 even on a very coarse mesh($N=5$).
Almost all the grid points are distributed in the inner layer of the solution.
Interestingly,  it seems that the nodes are dense where the second order derivative 
of the solution is  large.  This indicates a uniform distribution of the errors in $H^1$ norm.

 \begin{figure}[ht!]\label{fig:Example3}
 \centering
  \subfigure[$N=5$]{
   \includegraphics[width=0.22\textwidth]{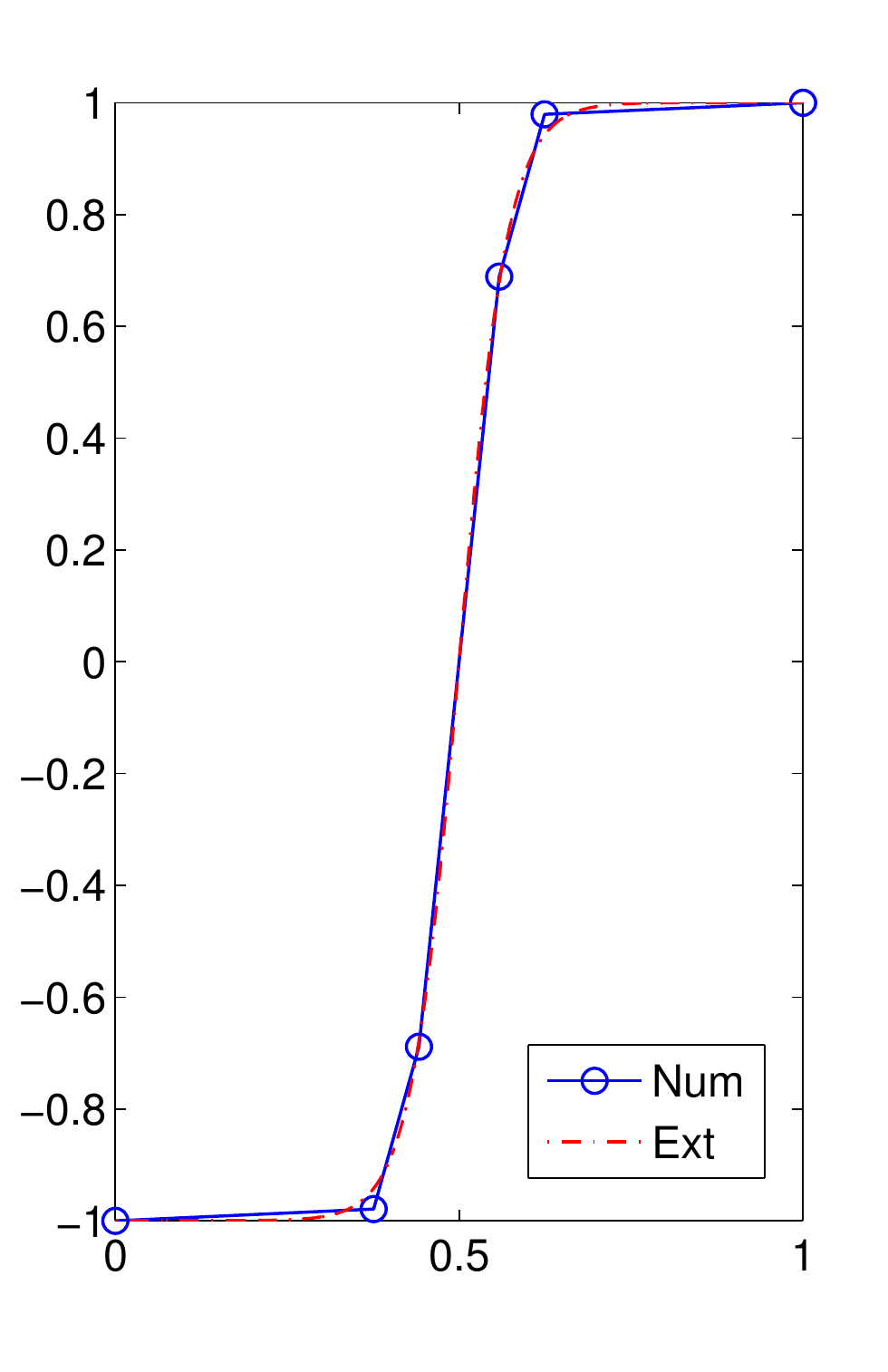}
}
  \subfigure[$N=10$]{
   \includegraphics[width=0.22\textwidth]{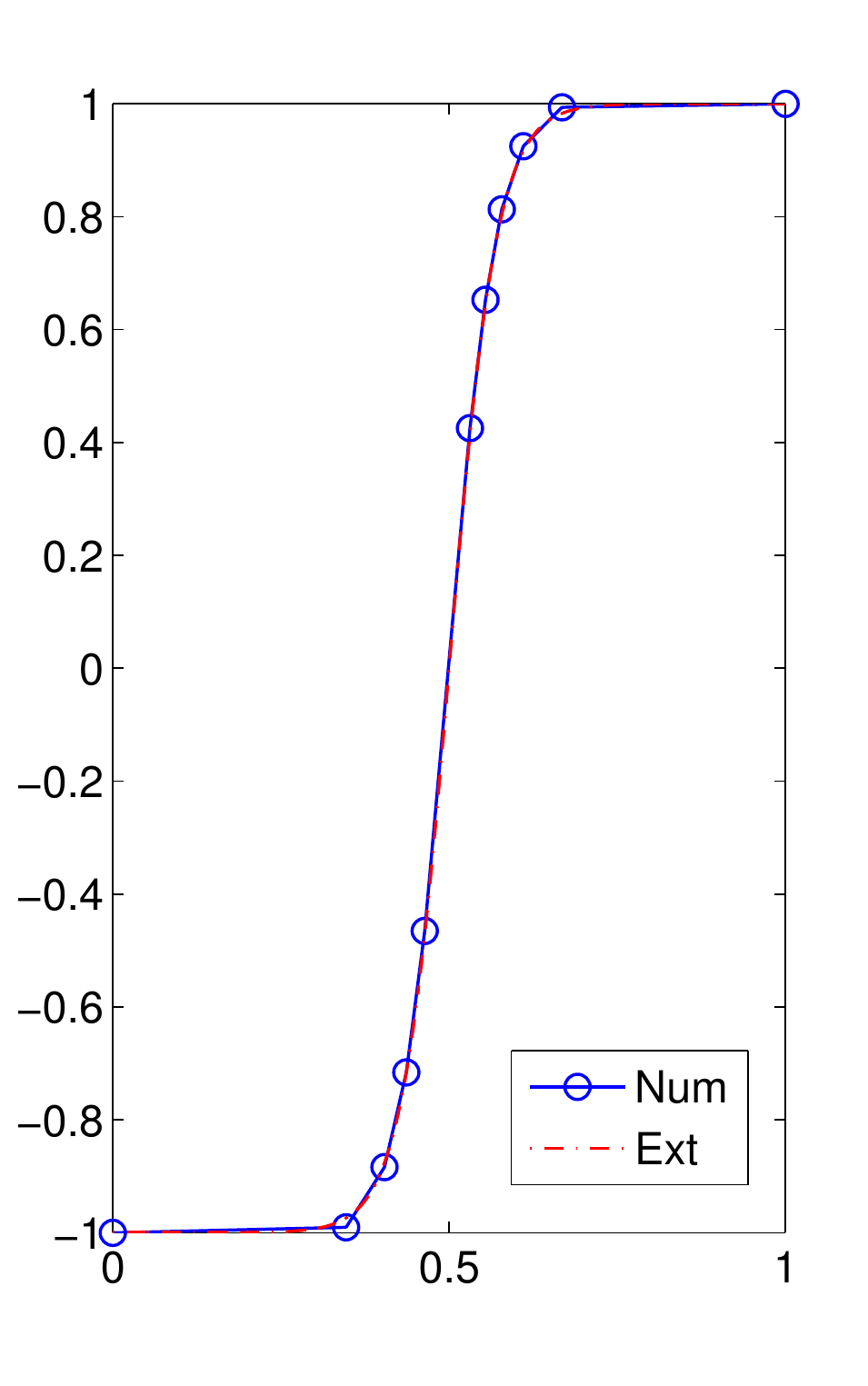}
}
  \subfigure[$N=20$]{
   \includegraphics[width=0.22\textwidth]{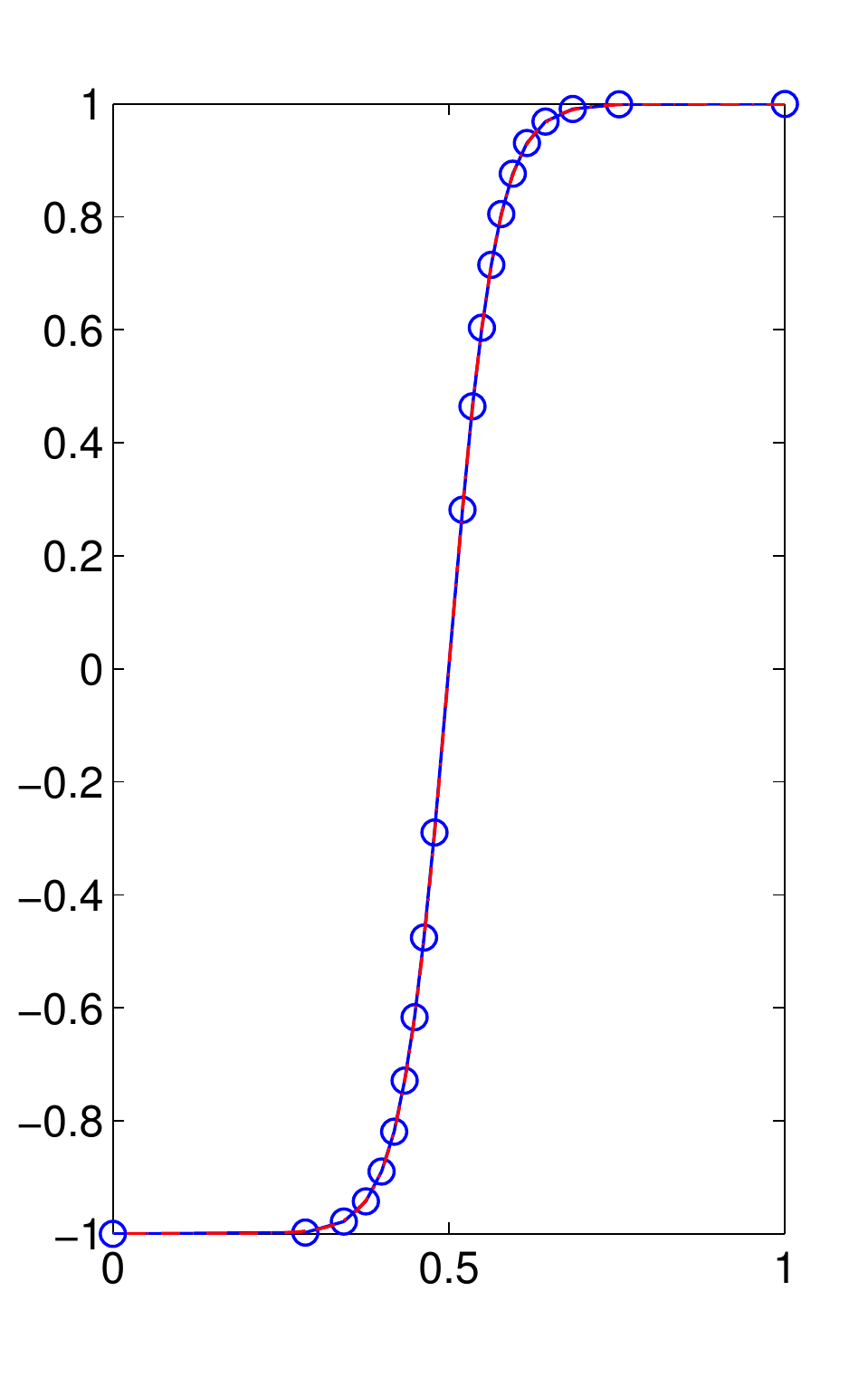}
}
  \subfigure[$N=40$]{
   \includegraphics[width=0.22\textwidth]{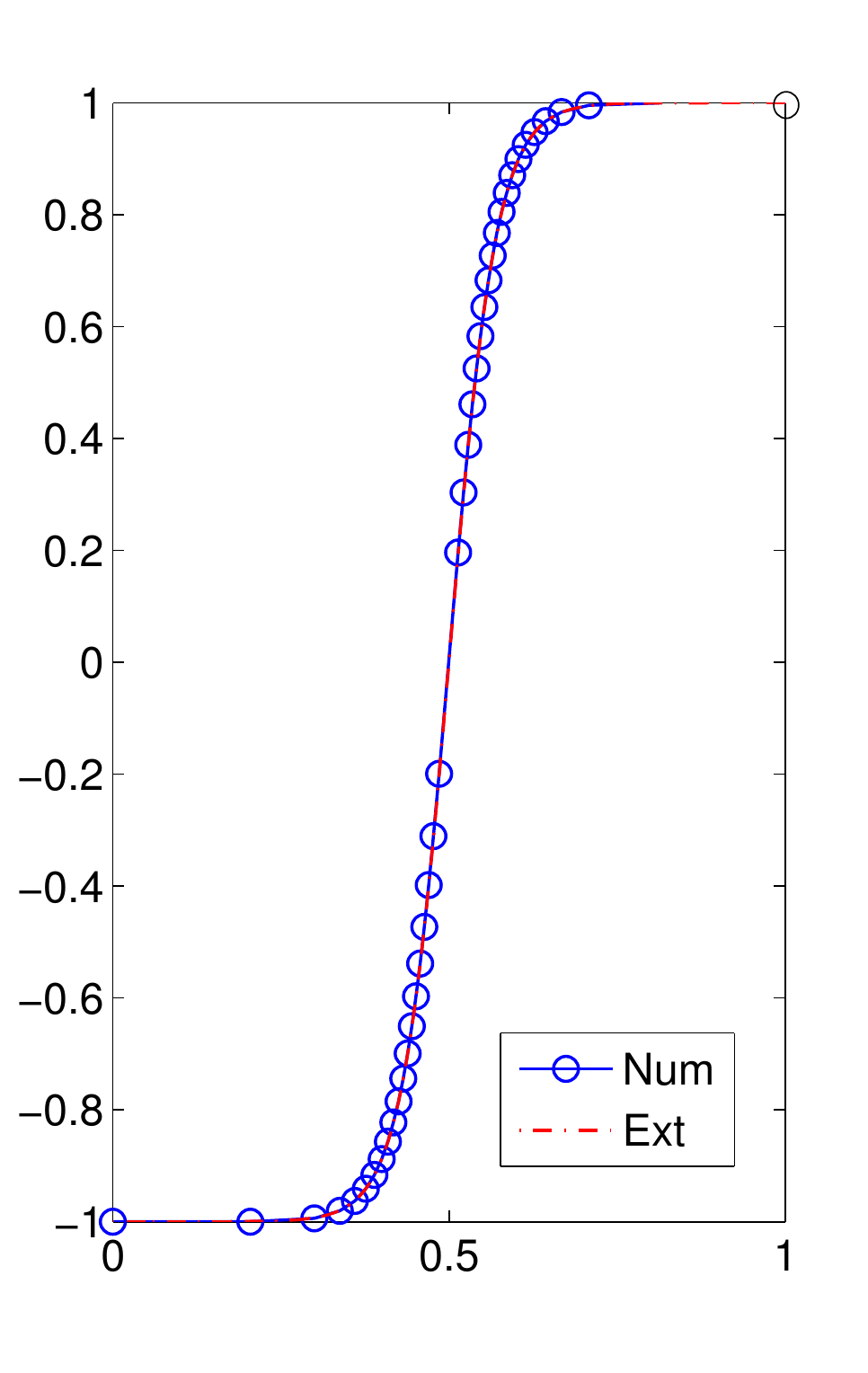}
}
 \caption{The numerical solutions of the Allen-Chan equation in stationary sate with $\eps=0.05$
 for various $N=5,10,20,40$.   The small circles in the figures mark the points $(x_k,\phi_{h,k})$.
 }
\end{figure}

In Table~\ref{tab:ErrorEx3}, we illustrate the $H^1$~error between $\phi_h^\infty$ and $\phi^{\infty}$:
\begin{equation*}
err_{H^1}:=\Big(\int_I(\partial_x \phi^\infty(x)-\partial_x \phi_h^{\infty}(x))^2 \dx\Big)^{1/2}.
\end{equation*}
It seems that the $H^1$~errors decrease with order $O(N^{-1})$. 
This is optimal for free-knot piecewise linear functions. 
This is consistent with the analytical results in the Section 4. 
In addition, we also compute the errors of the minimal energies by computing
\begin{equation*}
err_{eng}=|\E(\phi_h^\infty)-\sigma|,
\end{equation*}
where $\sigma=\frac{2\sqrt{2}}{3}$ is the total energy corresponding to the profile $\phi^{\infty}$.
We  see  that the errors of the energy decrease with order $O(N^{-2})$.
  
\begin{table}[h]\small
\caption{\small The $H^1$ errors of the stationary profile and the errors of the minimial energies in Example~3.}\label{tab:ErrorEx3}
\vspace{-0.2cm}
\begin{center}
\begin{tabular}{l|lc|lc}
\hline
{\bf Adaptive} & $err_{H^1}$  & order &  $err_{eng}$   & order \\
\hline
$N=5$&1.0004 & --  &0.0264 &-- \\
$N=10$ &0.5025& 0.99  & 0.00646&2.03\\
$N=20$ &0.2641 & 0.93 & 0.00175 &1.88 \\
$N=40$&0.1268& 1.06 &0.000402&2.12 \\
$N=80$&0.0691& 0.88&0.000120& 1.87 \\
\hline
\end{tabular}
\end{center}
\end{table}

{\bf Example 4.} In the last example, we test for a smaller $\eps=0.01$, which corresponds to a sharper inner layer. 
All other setups are the same as in the previous example. 

In Figure~\ref{fig:Example4}, we show the the stationary solutions of the Allen-Cahn equation for various different
choice of $N$. We see that the very sharp inner layer is resolved nicely by the moving finite element method even when
$N=5$. 
In Table~\ref{tab:ErrorEx4}, we show the $H^1$~errors between $\phi_h^{\infty}$ and $\phi^\infty(x)$
and also the errors of the minimal energy. We  see that the $H^1$ errors decrease with order $O(N^{-1})$
and the errors of the energies are of order $O(N^{-2})$. Both of them are optimal. 
This again verifies the theoretical results in the previous section.

 \begin{figure}[ht!]\label{fig:Example4}
 \centering
  \subfigure[$N=5$]{
   \includegraphics[width=0.22\textwidth]{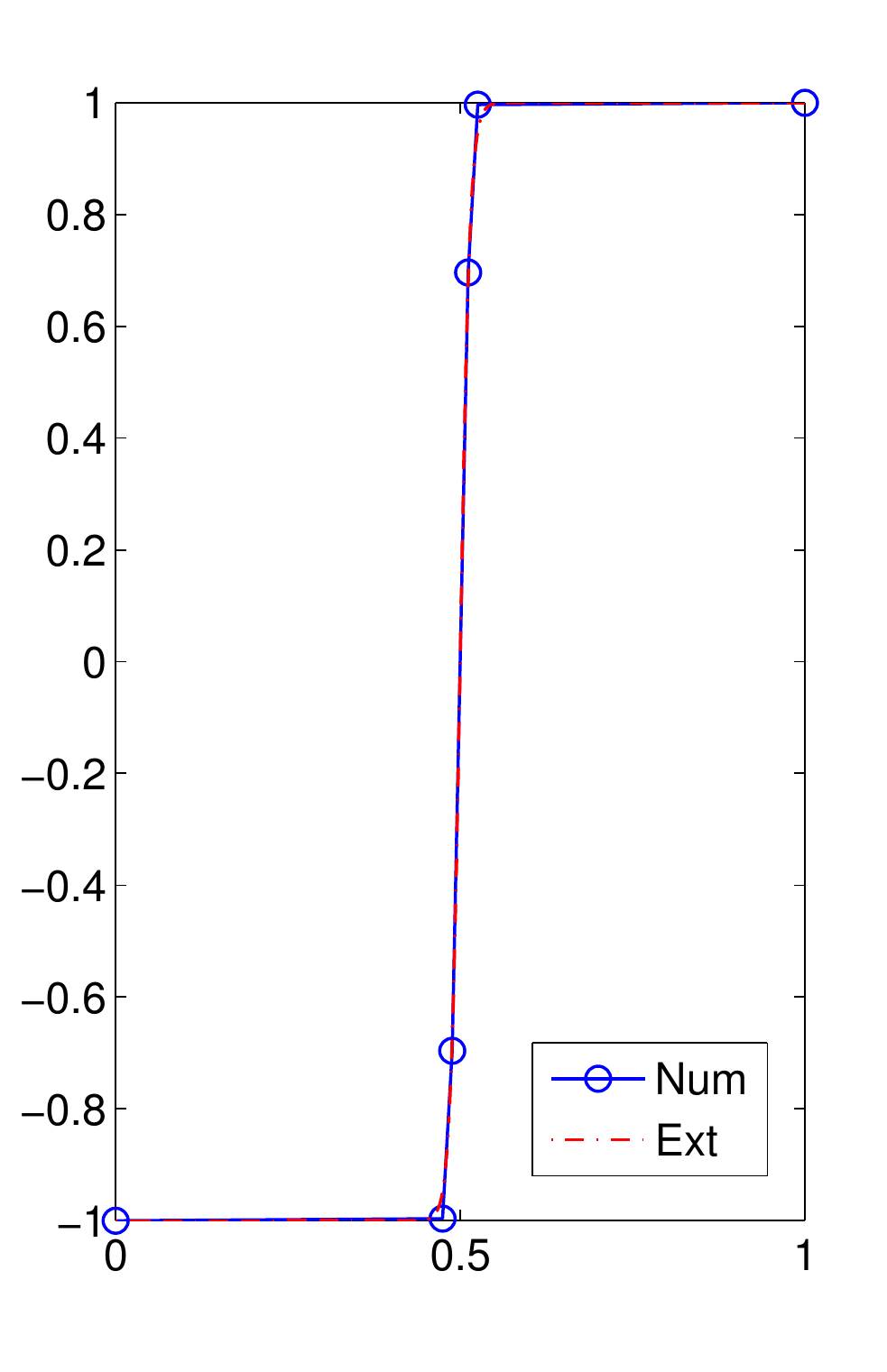}
}
  \subfigure[$N=10$]{
   \includegraphics[width=0.22\textwidth]{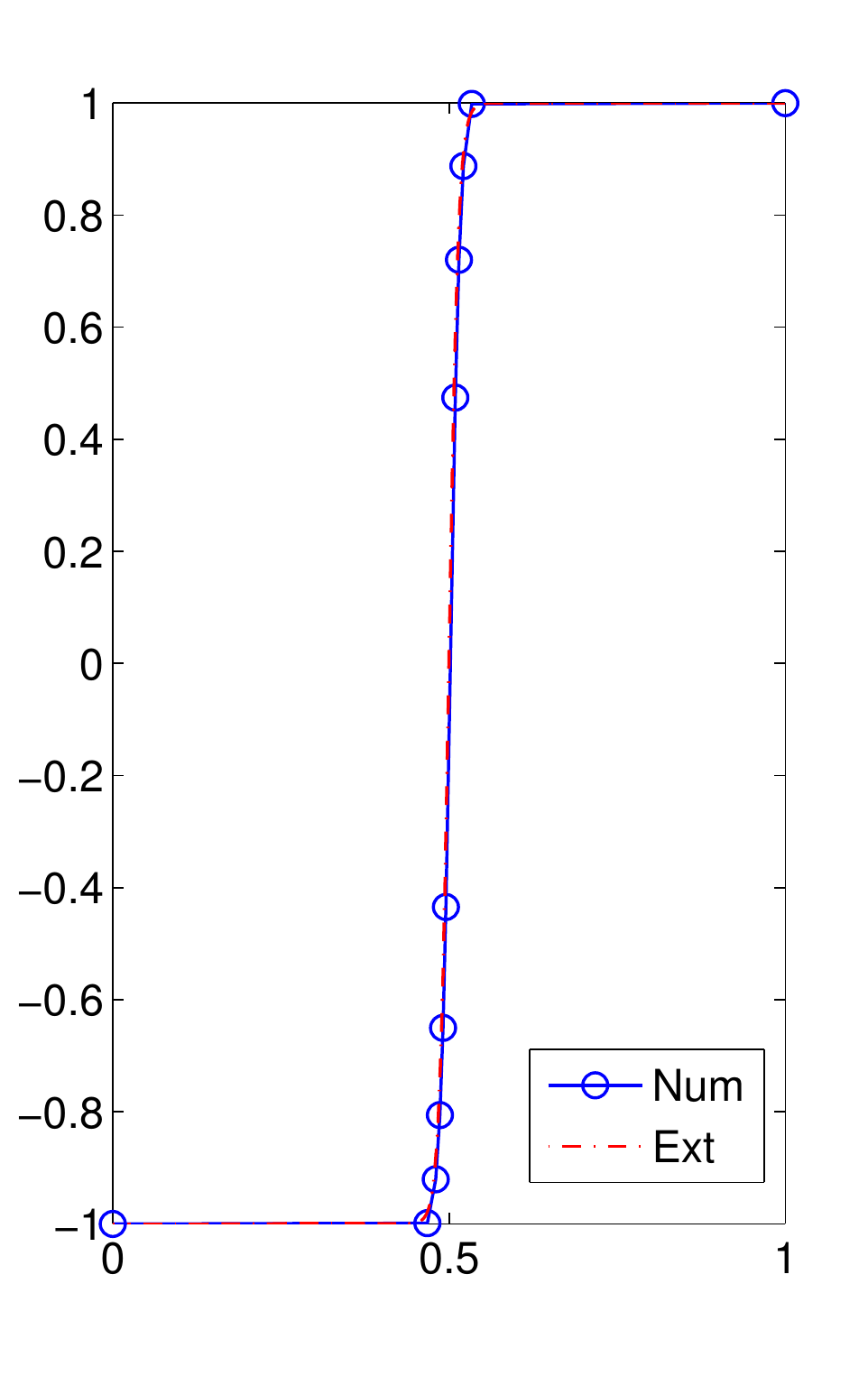}
}
  \subfigure[$N=20$]{
   \includegraphics[width=0.22\textwidth]{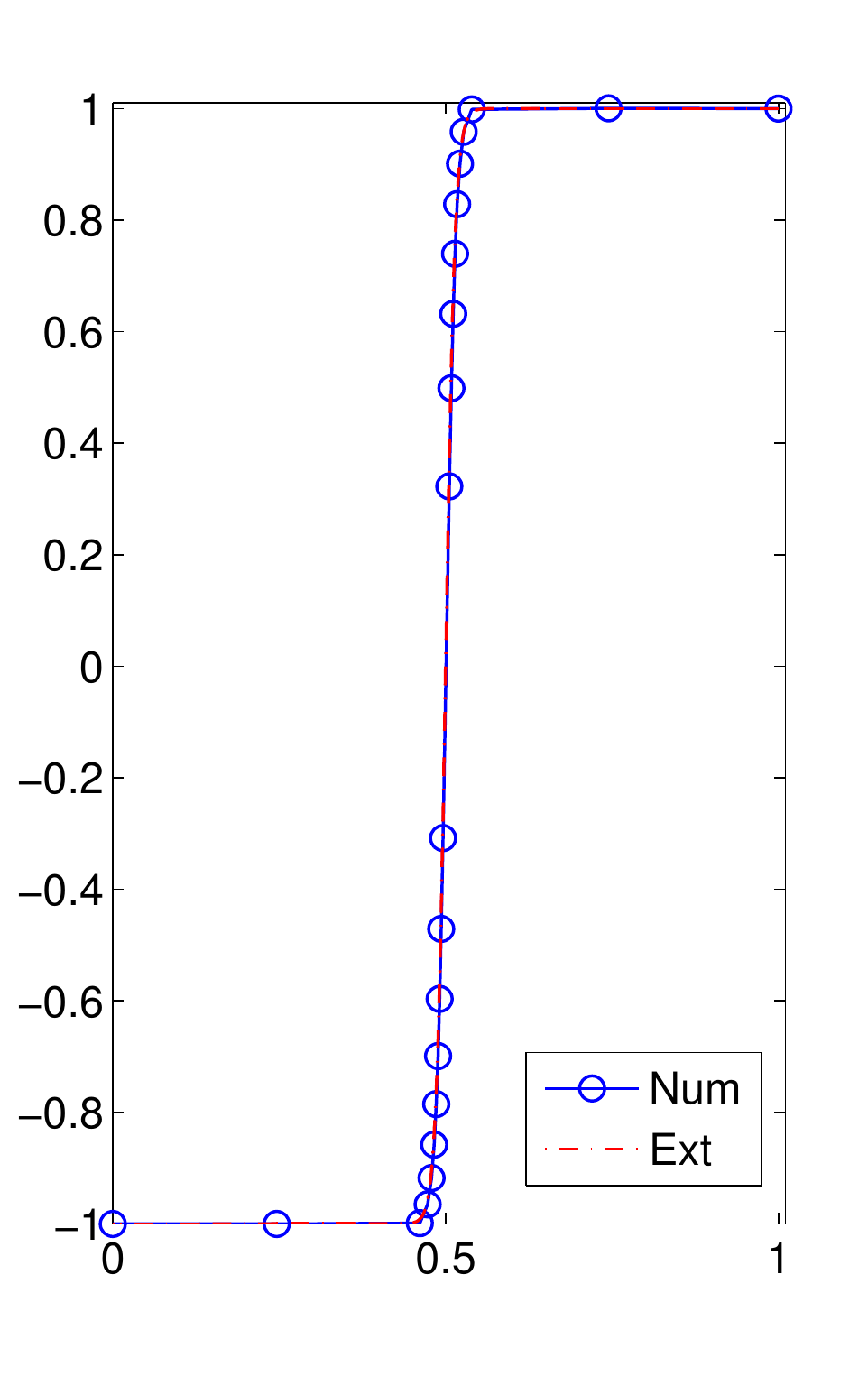}
}
  \subfigure[$N=40$]{
   \includegraphics[width=0.22\textwidth]{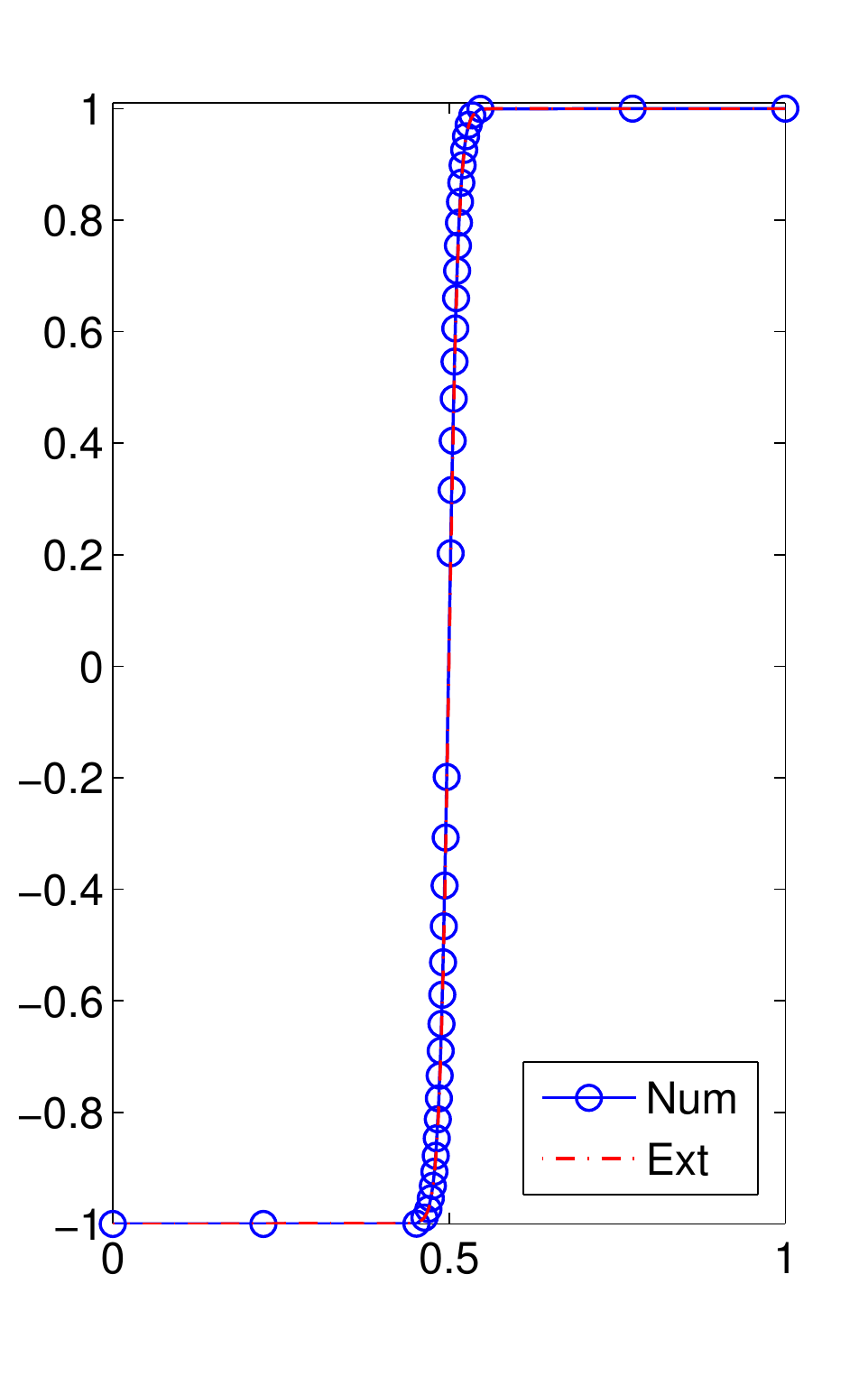}
}

 \caption{The numerical solutions of the Allen-Chan equation in stationary sate with $\eps=0.01$
 for various $N=5,10,20,40$.   The small circles in the figures mark the points $(x_k,\phi_{h,k})$.
 }
\end{figure}

\begin{table}[h]\small
\caption{\small The $H^1$ errors of the stationary profile and the errors of the minimial energies in Example~4.}\label{tab:ErrorEx4}
\vspace{-0.2cm}
\begin{center}
\begin{tabular}{l|lc|lc}
\hline
{\bf Adaptive} & $err_{H^1}$ & order &  $err_{eng}$   & order \\
\hline
$N=5$&2.2779 & --  &0.0281 &-- \\
$N=10$ &1.7365& 0.39 & 0.0067&2.07\\
$N=20$ &0.8043 &1.11 & 0.0017 &1.98 \\
$N=40$&0.3728& 1.11&0.000467&1.86 \\
$N=80$&0.1641& 1.18&0.000135& 1.79 \\
\hline
\end{tabular}
\end{center}
\end{table}

%
\section{Conclusions}
We show that the Onsager principle 
can act as a  variational framework for the MFEM for a gradient flow system.
The derivation of the method using the principle is much easier than the original approach in \cite{miller1981movingA}.
The discrete problem has the same energy dissipation structure as the
continuous one. 
This helps us to do numerical analysis for
a long time(stationary) solution of the gradient flow system. Under some conditions, 
we prove that the MFEM gives locally best approximation for the energy.
The optimal convergence rate can be obtained if a global minimizer is detected in
the free-knot piecewise linear function space. Although we consider only the one dimensional
case, the analysis can be generated to high dimensional
problems directly. In addition, it is also possible to consider more complicated gradient flows.

 In this paper, we restrict our analysis on  the stationary solution, although
 numerical experiments show that the method has optimal convergence rate
  for time dependent solutions
 (at least) for the linear diffusion problem.
Numerical analysis for the dynamic problems  will be left for future study. 
%
 

\section*{Acknowledgment}
This work was supported in part by NSFC grants DMS-11971469
 and  the National Key R\&D Program of China under Grant 2018YFB0704304 and Grant 2018YFB0704300.
\bibliographystyle{abbrv}
\bibliography{onsager.bib}
\bibliographystyle{abbrv}
\end{document}